\documentclass{amsart}
\usepackage{amssymb,amsmath,amsfonts,amsthm,amsxtra,afterpage}
\usepackage[dvips,all]{xy}
\usepackage{xcolor}
\usepackage{mdframed}
\usepackage[normalem]{ulem}
\usepackage{enumerate}
\usepackage{enumerate}

\usepackage{todonotes}

\newtheorem{thm}{Theorem}[section]
\newtheorem{prop}[thm]{Proposition}
\newtheorem{lem}[thm]{Lemma}

\newtheorem{cor}[thm]{Corollary}

\theoremstyle{definition}
\newtheorem{definition}[thm]{Definition}
\newtheorem{eg}[thm]{Example}

\theoremstyle{remark}
\newtheorem{rmk}[thm]{Remark}

\numberwithin{equation}{section}

\renewenvironment{proof}[1][\proofname]{\begin{trivlist}\item[\hskip \labelsep \itshape \bfseries #1{}\hspace{2ex}]}
{\qed\end{trivlist}}

\begin{document}
\title[Weitzenb\"{o}ck]{Invariant Rings of $ \mathbb{G}_{a} $-Representations are not always Finitely Generated in Positive Characteristic.}
\author{Stephen Maguire}
\email{maguire2@illinois.edu}

\begin{abstract}
    Hilbert's 14th Problem asks the following question.  Given a linear representation $ \beta: G \to \operatorname{GL}(\mathbf{V}) $ of a linear algebraic group over a field $ k $ is the ring $ S_{k}(\mathbf{V}^{\ast}) $ a finitely generated $ k $-algebra?  For reductive groups the answer is yes.  However, in general the answer is no.  Nagata provided one of the earliest counterexamples to this claim in \cite{Nagata} and his counterexample was extended by Shigeru Mukai in \cite{MukaiCounterexample}.  However, if $ G $ is equal to $ \mathbb{G}_{a} $ and the characteristic of $ k $ is equal to zero, then the answer to Hilbert's 14th problem is yes.  Roland Weitzenb\"{o}ck first proved this result in 1932 in an article in Acta Mathematica (see \cite{Weitz}).  Seshadri gave a more accessible proof in \cite{Seshadri}.

    While Roland Weitzenb\"{o}ck did not conjecture this claim, the question of whether the theorem that bears his name still holds if the characteristic of the base field $ k $ is $ p>0 $ is known as ``the Weitzenb\"{o}ck conjecture''.  We aim to use Mukai's strategy to give a counterexample to the Weitzenb\"{o}ck conjecture.  Namely, we construct a six dimensional representation over a field of positive characteristic such that the invariant ring is isomorphic to the Cox ring of the blow-up of a toric surface at the identity of the torus.  We use the geometry of the underlying toric variety to show that this Cox ring is not finitely generated.
\end{abstract}
\maketitle
\section{An Outline of Our Strategy.}
We would like to talk about our strategy for producing a counterexample to the Weitzenb\"{o}ck Conjecture.  This strategy is very much inspired by Shigeru Mukai's paper \cite{MukaiCounterexample} in which he describes an alternate proof of Nagata's counterexample and proves a much broader result.  In this paper, Mukai starts with the following $ 2n $-dimensional representation $ \beta $ of $ \mathbb{G}_{a}^{n} \cong \operatorname{Spec}(k[T]) $ on $ \mathbb{A}^{2n}_{k} \cong \operatorname{Spec}(k[X]) $ whose co-action is described below:
\begin{align*}
    x_{i} & \mapsto x_{i} \quad 1 \le i \le n \\
    x_{n+i} & \mapsto x_{n+i}+t_{i} x_{i} \quad 1 \le i \le n.
\end{align*}
Mukai then takes $ n $-points $ c_{i} = (c_{1,i},\dots,c_{s,i}) $ of $ \mathbb{A}^{s}_{k} $ in general position.  The sub-group $ \mathbb{G}_{a}^{n-s} \cong \mathcal{V}(\langle \sum_{i=1}^{n} c_{j,i}t_{i} \rangle_{j=1}^{s}) $ acts on $ \mathbb{A}^{2n}_{k} $ via the induced action.  If $ \widehat{x_{i}} $ is equal to $ \prod_{j \ne i, 1 \le j \le n} x_{j} $, and $ r_{1}(X),\dots,r_{s}(X) $ are the invariant functions described below:
\begin{equation*}
    r_{j}(X) = \sum_{i=1}^{n} c_{j,i}x_{n+i} \widehat{x_{i}}
\end{equation*}
then
\begin{equation*}
    k[x_{1}^{\pm 1},\dots,x_{n}^{\pm 1},x_{n+1},\dots,x_{2n}]^{\mathbb{G}_{a}^{n-s}} = k[x_{1}^{\pm 1},\dots,x_{n}^{\pm 1}, r_{1}(X),\dots,r_{s}(X)].
\end{equation*}
Mukai then proves that $ k[X]^{\mathbb{G}_{a}^{n-s}} $ is $ \mathbb{Z}^{n+1} $-multigraded.  Let $ B_{E} $ be the vector space generated by homogeneous, degree $ e_{n+1} $, polynomials $ G_{e_{n+1}}(y_{1},\dots,y_{s}) $ such that $ x_{i}^{e_{i}} $ divides $ G_{e_{n+1}}(r_{1}(X),\dots,r_{s}(X)) $ for $ 1 \le i \le n $.  The vector space $ B_{E} $ corresponds to a sub-space of invariant polynomials.  Namely, for each $ G_{e_{n+1}}(Y) \in B_{E} $, one obtains an invariant polynomial $ G_{e_{n+1}}(r_{1}(X),\dots,r_{s}(X))/\left(\prod_{i=1}^{n} x_{i}^{e_{i}}\right) $.  Mukai showed that $ k[X]^{\mathbb{G}_{a}^{n-s}} \cong \oplus_{E} B_{E} $.

If $ \mathfrak{p}_{i} $ is the ideal $ \langle x_{i} \rangle k[X] \cap k[r_{1}(X),\dots,r_{s}(X)] $, then the condition that $ x_{i}^{e_{i}} $ divide $ G_{e_{n+1}}(r_{1}(X),\dots,r_{s}(X)) $ is equivalent to the condition that $ G_{e_{n+1}}(R) \in \mathfrak{p}_{i}^{(e_{i})} $.  Moreover, if one imposes the grading on $ k[r_{1}(X),\dots,r_{s}(X)] $ where $ \deg(r_{i}(X)) $ is equal to one, then the ideals $ \mathfrak{p}_{i} $ are homogeneous.  If $ k[r_{1}(X),\dots,r_{s}(X)] $ is the homogeneous coordinate ring of $ \mathbb{P}^{s-1}_{k} $, then there is a point $ u_{i} $ corresponding to $ \mathfrak{p}_{i} $ via $ u_{i} = \mathcal{V}(\mathfrak{p}_{i}) $.  The points $ u_{1},\dots,u_{n} $ of $ \mathbb{P}^{s-1}_{k} $ are in general position.  Let $ W $ be the blow-up of $ \mathbb{P}^{s-1}_{k} $ at these $ n $-points, let $ E_{i} $ be the exceptional divisor corresponding to the blow-up at $ u_{i} $, and let $ \pi: W \to \mathbb{P}^{s-1}_{k} $ be the natural projection map.  The condition that $ G(R) \in \mathfrak{p}_{i}^{(e_{i})} $ for a degree $ e_{n+1} $ polynomial $ G(Y) $ is equivalent to the condition that $ \pi^{\ast}(G(Y)) \in H^{0}(W, \mathcal{O}_{W}(e_{n+1}H-\sum_{j=1}^{n} e_{i}E_{i})) $.  Thus the Cox ring of $ W $ is isomorphic to the ring of invariants $ k[X]^{\mathbb{G}_{a}^{n-s}} $.  Mukai then proves that the Cox ring of $ W $ is not finitely generated using the geometry of $ W $.  We will use a different strategy than the one he uses to prove that the corresponding Cox ring in our proof is not finitely generated.

In our paper we first show in Lemma ~\ref{L:multiGrading} (see page \pageref{L:multiGrading}) that there is always at least one multi-grading for the ring of invariants.  We then create a six dimensional representation $ \beta: \mathbb{G}_{a} \to \operatorname{GL}(\mathbf{V}) $ such that if $ \{x_{1},\dots,x_{6}\} $ is a basis of $ \mathbf{V}^{\ast} $, then there are four bi-homogeneous polynomials $ x_{3},r_{1}(X),r_{2}(X),r_{3}(X) $ such that if $ A $ is equal to $ k[x_{3},r_{1}(X),r_{2}(X),r_{3}(X)] $ (the ring $ A $ does not contain $ x_{1} $ and $ x_{1} $ does not divide $ r_{1}(X),r_{2}(X) $ or $ r_{3}(X) $), then $ \left(A[x_{1}]\right)_{x_{1}} \cong k[X]_{x_{1}}^{\mathbb{G}_{a}} $ and $ A $ is a four dimensional, bi-graded, polynomial ring.

We then construct a toric variety $ X_{\Sigma} $ such that $ \operatorname{Cox}(X_{\Sigma}) \cong A $ with the exact same bi-homogeneous grading.  First we adjust the $ 2 \times 4 $ matrix of weights of $ x_{3},r_{1}(X),r_{2}(X),r_{3}(X) $ to obtain a matrix $ C $.  Then we compute the kernel $ B $ of the exact sequence below:
\begin{equation*}
\xymatrix{
    0 \ar[r] & \mathbb{Z}^{2} \ar[r]^{B} & \mathbb{Z}^{4} \ar[r]^{C} & \mathbb{Z}^{2} \ar[r] & 0.
    }.
\end{equation*}
The columns of $ B $ define rays of a fan $ \Sigma $ in $ \mathbb{R}^{2} $.  From this fan we recover a toric variety $ X_{\Sigma} $ and prove that $ X_{\Sigma} \cong \left(\mathbb{A}^{4}_{k} \setminus Z(\Sigma)\right)//\mathbb{G}_{m}^{2} $ for a certain variety $ Z(\Sigma) $ (see \cite[Chapter 5, Homogeneous Coordinates on Toric Varieties, Section 1, Quotient Constructions of Toric Varieties, Lemma 5.1.1]{CoxLittleSchenck} regarding this construction).  This construction proves that, after a suitable adjustment of the weights, the Cox ring of $ X_{\Sigma} $ is isomorphic to $ A $ as a graded ring.

If $ y_{1},\dots,y_{4} $ are the elements below:
\begin{align*}
    y_{1} &= x_{3} \\
    y_{i} &= r_{i-1}(X) \quad 2 \le i \le 4.
\end{align*}
then $ \mathfrak{p} = \langle x_{1} \rangle \cap A $ is a multi-homogeneous ideal such that $ \mathcal{V}(\mathfrak{p}) $ is the identity $ e $ of the torus with its identification via the fan $ \Sigma $.  Let $ B_{d,I} $ be the vector sub-space of polynomials $ G_{I}(Y) $ of bi-degree $ I $, such that $ x_{1}^{d} $ divides $ G_{I}(Y) $.  We show in Theorem ~\ref{T:CoxRingStart} that $ k[X]^{\mathbb{G}_{a}} $ is isomorphic to $ \oplus_{d,I} B_{d,I} $.

If $ \mathfrak{p} $ is the ideal $ \langle x_{1} \rangle \cap k[Y] $, then the condition that $ x_{1}^{d} $ divides $ G_{I}(Y) $ is equivalent to the condition that $ G_{I}(Y) \in \mathfrak{p}^{(d)} $.  The ideal $ \mathfrak{p} $ is also homogeneous with respect to the bi-grading and $ \mathcal{V}(\mathfrak{p}) $ corresponds to the point $ e $ of $ X_{\Sigma} $.  If $ W $ is the blow-up of $ X_{\Sigma} $ at $ e $ and $ \pi: W \to X_{\sigma} $ is the natural projection, then the condition that $ G_{I}(Y) \in \mathfrak{p}^{(d)} $ is equivalent to the condition that $ \pi^{\ast}(G_{I}(Y)) \in H^{0}(W, \mathcal{O}_{W}(D_{I}-dE)) $.  Here $ D_{I} $ is a divisor in $ X_{\Sigma} $ of bi-degree $ I $ such that $ G_{I}(Y) \in H^{0}(X_{\Sigma}, \mathcal{O}_{X_{\Sigma}}(D)) $.  So, the Cox ring of $ W $ is isomorphic to the ring of invariants $ k[X]^{\mathbb{G}_{a}} $.  We then prove that the Cox ring of $ W $ is not finitely generated by proving that the pseudo-effective cone contains an extremal, irrational ray.

To construct the localization of our ring of invariants we use theory from Section ~\ref{S:Reps}.  We then relate this to the Cox ring of a specific toric variety using results from Section ~\ref{S:TA}.  We then describe the properties of this toric variety in Section ~\ref{S:specificTor} and prove the main result in Theorem ~\ref{S:Main}.  Pre-requisite information on Cox rings, positivity properties of divisors and toric varieties are included in appendices for the purpose of this paper being self contained.

\section{Beginning Information on $ \mathbb{G}_{a} $-Representations and $ \mathbb{G}_{a} $-Actions.} \label{S:Reps}
\begin{definition}
    An \emph{additive polynomial} $ b(t) $ is a polynomial of $ k[t] $ such that $ b(t_{1}+t_{2})=b(t_{1})+b(t_{2}) $.
\end{definition}
If the characteristic of $ k $ is zero, then the only additive polynomials are constant multiples of $ t $.  It is important to note that an additive polynomial $ c(t) $ is separable if and only if it contains a non-zero linear term.
\begin{definition} \label{D:ore}
    Let $ k $ be a field of positive characteristic $ p > 0 $.  The \emph{Ore ring} is the ring whose underlying set is the set of additive polynomials of $ k[t] $.  Addition is pointwise addition in the Ore ring, while multiplication is composition.
\end{definition}
There is another way to view the Ore ring.  If $ F^{\ell} $ corresponds to the $ \ell $-th iterate of the Frobenius morphism, then $ t^{p^{\ell}} = F^{\ell}(t) $.  The underlying set of the Ore ring is $ k[F] $, and addition is the same as in the polynomial ring $ k[F] $.  However, $ aF^{j} \cdot (b F^{i}) = a b^{p^{j}} F^{i+j} $.  If $ b(F) \in k[F] $, then the map which sends $ b(F) $ to $ b(F) \circ t $ is an isomorphism between our two conceptualizations of the Ore ring.  We shall denote the Ore ring by $ \mathfrak{O} $.

By an analogous proof to the case of $ k[t] $, there is a right division algorithm for the Ore ring.  If $ k $ is a perfect field, then there is a left division algorithm (see \cite[Proposition 1.6.2, Proposition 1.6.5]{Goss} and \cite[Chapter I, Theorem 1, pg. 562]{Ore}).  As a result, the Ore ring is a non-commutative Euclidean ideal domain.

If $ k $ is algebraically closed, then $ k $ is perfect.  So for $ k $ algebraically closed, the Ore ring is a non-commutative Euclidean Ring, and so for any two additive polynomials $ c_{1}(t) $ and $ c_{2}(t) $ a greatest common divisor exists.  We will denote the greatest common divisor of $ c_{1}(t) $ and $ c_{2}(t) $ in $ \mathfrak{O} $ by $ \mathfrak{O}(c_{1}(t),c_{2}(t)) $.  For two additive polynomials $ c_{1}(t) $ and $ c_{2}(t) $, if $ b(t) $ is equal to $ \mathfrak{O}(c_{1}(t),c_{2}(t)) $, then there are additive polynomials $ d_{1}(t) $ and $ d_{2}(t) $ such that:
\begin{align*}
    d_{1}(b(t)) &= c_{1}(t) \\
    d_{2}(b(t)) &= c_{2}(t).
\end{align*}
\begin{eg} \label{EG:boft}
    Let $ k $ be an algebraically closed field of characteristic $ p>0 $ and let $ b(t) $ be an additive polynomial.  The sub-scheme $ \mathcal{V}(\langle b(t) \rangle) $ inherits the structure of an affine group scheme from the Hopf algebra morphisms of $ \mathbb{G}_{a} $.  We denote this group scheme by $ \mathbf{ker}(b(t)) $.  Note that this group scheme may not be reduced.
\end{eg}
\begin{definition} \label{D:aPair}
    Let $ b(t) $ be an additive polynomial of $ k[t] $, and let $ \operatorname{Spec}(A) $ be a $ \mathbb{G}_{a} $-variety with action $ \beta $.  If there are elements $ g \in A $ and $ h \in A^{\mathbb{G}_{a}} $, such that $ \beta^{\sharp}(g) = g+b(t) h $, then $ (g, h) $ is \emph{a $ b(t) $-pair}.  Another way to say this is that if $ x \in \operatorname{Spec}(A) $ and $ t_{0} \in \mathbb{G}_{a} $, then
    \begin{equation*}
        g(t_{0} \ast x) = g(x)+b(t_{0})h(x).
    \end{equation*}
    If $ b(t) $ is equal to $ t $, then we say that $ (g,h) $ \emph{is a principle pair}.  A $ b(t) $-pair $ (g,h) $ is \emph{trivial} if $ hb(t) $ is equal to zero.
\end{definition}
\begin{definition} \label{D:principleAPair}
    If $ \operatorname{Spec}(A) $ is a $ \mathbb{G}_{a} $-variety over a field $ k $ with action $ \beta $, then a non-trivial $ b(t) $-pair $ (g,h) $ is \emph{a quasi-principle $ b(t) $-pair} if $ \mathbf{ker}(b(t)) $ acts trivially on all of $ \operatorname{Spec}(A) $.
\end{definition}
\begin{definition} \label{D:plinth1}
    If $ \operatorname{Spec}(A) $ is a $ \mathbb{G}_{a} $-variety, then the \emph{pedestal ideal} is the ideal of $ A $ generated by $ \{ 0 $ and all $ h $ such that there exists a non-zero, additive polynomial $ b(t) $ and a $ g\in A $ such that $ (g,h) $ is a quasi-principle $ b(t) $-pair $ \} $ (see Definition ~\ref{D:aPair} on page \pageref{D:aPair}).  Denote this ideal by $ \mathfrak{P}(A) $ or just $ \mathfrak{P} $ when the underlying ring is clear.  The \emph{pedestal scheme} is $ \mathcal{V}(\mathfrak{P}(A)) \subseteq \operatorname{Spec}(A) $.  A point $ x \in \operatorname{Spec}(A) $ is \emph{affine stable} if $ x \in \operatorname{Spec}(A) \setminus \mathcal{V}(\mathfrak{P}(A)) $.  We denote the sub-variety of affine stable points of $ \operatorname{Spec}(A) $ by $ \operatorname{Spec}(A)^{as} $.  The variety $ \operatorname{Spec}(A) $ is quasi-principle if there is a quasi-principle pair $ (g,h) $.

    If $ \operatorname{Spec}(A) $ is a $ \mathbb{G}_{a} $-variety, then the \emph{large pedestal ideal} is the ideal of $ A $ generated by $ \{ h $ such that there exists a non-zero, additive polynomial $ c(t) $ and a $ g \in A $ such that $ (g,h) $ is a $ c(t) $-pair $ \} $.  We denote the large pedestal ideal of $ A $ by $ \mathfrak{P}_{g}(A) $.
\end{definition}
If $ \operatorname{Spec}(A) $ is quasi-principle variety with action $ \beta $, and $ (g,h) $ is a quasi-principle $ b(t) $-pair, then $ \mathbf{ker}(b(t)) $ acts trivially on $ \operatorname{Spec}(A) $.  Since
\begin{align*}
    A^{\mathbb{G}_{a}} &= (A^{\mathbf{ker}(b(t))})^{(\mathbb{G}_{a}//\mathbf{ker}(b(t)))} \\
    &= A^{\mathbb{G}_{a}//\mathbf{ker}(b(t))},
\end{align*}
it often makes sense to replace $ \mathbb{G}_{a} $ by $ \mathbb{G}_{a}//\mathbf{ker}(b(t)) $ and assume that $ (g,h) $ is a principle pair.

The pedestal ideal is related to the concept of the plinth ideal.  The plinth ideal has been studied in characteristic zero in the situation where a $ \mathbb{G}_{a} $-action is determined by a locally nilpotent derivation.  If the ring $ A^{\mathbb{G}_{a}} $ is a finitely generated ring over a field of characteristic zero, then the plinth ideal is equal to $ \mathfrak{P}(A) \cap A^{\mathbb{G}_{a}} $.
\begin{prop} \label{P:expandedSliceSep}
    If $ \operatorname{Spec}(A) $ is a $ \mathbb{G}_{a} $-variety with action $ \beta $, such that $ A $ is generated by
    $ z_{1},\dots,z_{n} $ as a $ k $-algebra, $ \beta^{\sharp}(z_{i}) $ is equal to $ v_{i}(Z,t) $, and $ (g,h) $ is a principle pair (see Definition ~\ref{D:aPair}), then the rational functions:
    \begin{equation} \label{E:82}
        f_{i}(Z) = v_{i}(z_{1},\dots,z_{n},-g/h) \quad 1 \le i \le n
    \end{equation}
    are invariant.  Moreover, the map below:
    \begin{equation*}
        r(Z) \mapsto r(f_{1}(Z),\dots,f_{n}(Z)),
    \end{equation*}
    is a homomorphism from $ A_{h} $ to $ A_{h}^{\mathbb{G}_{a}} $.
\end{prop}
\begin{proof}
    If $ x \in \operatorname{Spec}(A) $ and $ s,w \in \mathbb{G}_{a} $, then since $ \operatorname{Spec}(A) $ is a $ \mathbb{G}_{a} $-variety:
    \begin{align}
        v_{i}(x,s+w) &= z_{i}((s+w) \ast x), \notag \\
        &= z_{i}(s \ast (w \ast x)), \notag \\
        &= v_{i}(z_{1}(w \ast x),\dots,z_{n}(w \ast x),s). \label{E:87}
    \end{align}
    Equating the left and right hand sides of ~\eqref{E:87}:
    \begin{equation} \label{E:88}
        v_{i}(x,s+w) = v_{i}(z_{1}(w \ast x),\dots,z_{n}(w \ast x),s).
    \end{equation}
    If we substitute $ (-g/h)(w \ast x) $ for $ s $, then:
    \begin{align*}
        f_{i}(w \ast x) &= v_{i}(z_{1}(w \ast x),\dots,z_{n}(w \ast x),(-g/h)(w \ast x)) \\
        &= v_{i}(z_{1}(w \ast x),\dots,z_{n}(w \ast x),(-g/h)(x)-w) \\
        &= v_{i}(x,(-g/h)(x)-w+w) \\
        &= v_{i}(x,(-g/h)(x)) \\
        &= f_{i}(x),
        \end{align*}
    where the jump from the second line to the third uses ~\eqref{E:88}.  Because $ f_{i}(Z) $ is constant on the orbits of $ \mathbb{G}_{a} $, it is in $ A_{h}^{\mathbb{G}_{a}} $.  Now notice that the map which sends $ r(Z) $ to $ r(f_{1}(Z),\dots,f_{n}(Z)) $ is the composition of the co-action $ \beta^{\sharp} $ with the evaluation morphism which sends $ t $ to $ -g/h $.  Therefore, this is a homomorphism.
\end{proof}
\begin{prop} \label{P:itsTheRing}
    Let $ \operatorname{Spec}(A) $ be a $ \mathbb{G}_{a} $-variety with action $ \beta $ and let \linebreak $ z_{1},\dots,z_{n} $ generate $ A $ as a $ k $-algebra.  If $ (g, h) $ is a principle pair (see Definition ~\ref{D:aPair} on page \pageref{D:aPair}), $ v_{i}(Z,t) $ is equal to $ \beta^{\sharp}(z_{i}) $, and $ f_{i}(Z) $ equals $ v_{i}(Z,-g/h) $, then
    \begin{equation*}
        k[f_{1}(Z),\dots,f_{n}(Z)]_{h} = A_{h}^{\mathbb{G}_{a}},
    \end{equation*}
    i.e., the homomorphism which sends $ r(Z) \in A_{h} $ to $ r(f_{1}(Z),\dots,f_{n}(Z)) $ is surjective and if $ r(Z) $ is an element of $ A^{\mathbb{G}_{a}}_{h} $, then it is equal to $ r(f_{1}(Z),\dots,f_{n}(Z)) $.
\end{prop}
\begin{proof}
    If $ (g,h) $ is a principle pair, $ r(Z) \in (A_{h})^{\mathbb{G}_{a}} $, and we denote $ \beta^{\sharp}(z_{i}) $ by $ v_{i}(Z,t) $, then:
    \begin{align}
        r(Z) &= \beta^{\sharp}(r(Z)), \notag \\
        &= r(\beta^{\sharp}(z_{1}),\dots,\beta^{\sharp}(z_{n})), \notag \\
        &=r\left( v_{1}(Z,t),\dots,v_{n}(Z,t) \right). \label{E:36}
    \end{align}
    Since $ r(Z) $ does not depend on the value of $ t $ in ~\eqref{E:36}, we may substitute $ (-g/h) $ for $ t $ in ~\eqref{E:36}.  So, we may write $ r(Z) $ as follows:
    \begin{align*}
        r(Z) &= r(v_{1}(Z,-g/h),\dots,v_{n}(Z,-g/h)) \\
        &=r(f_{1}(Z),\dots,f_{n}(Z)).
    \end{align*}
    Therefore $ r(Z) \in k[f_{1}(Z),\dots,f_{n}(Z)]_{h} $, i.e.,
    \begin{equation*}
        k[f_{1}(Z),\dots,f_{n}(Z)]_{h} = A^{\mathbb{G}_{a}}_{h}.
    \end{equation*}
\end{proof}
\begin{prop} \label{T:qJustification}
    Let $ \beta: \mathbb{G}_{a} \to \operatorname{GL}(\mathbf{V}) $ be an $ n $-dimensional, non-trivial, quasi-principle, linear representation of $ \mathbb{G}_{a} $, let $ \{x_{1},\dots,x_{n}\} $ be a basis of $ \mathbf{V}^{\ast} $, and let $ (g(X),h(X)) $ be a non-trivial principle pair such that $ g(X),h(X) \in \mathbf{V}^{\ast} $.  After a change of basis, we may assume that $ h(X) $ is equal to $ x_{1} $ and $ k[X]^{\mathbb{G}_{a}}_{x_{1}} $ is isomorphic to the localization of a polynomial ring in $ n-1 $ variables with respect to one of the indeterminates.
\end{prop}
\begin{proof}
    If $ g(X) $ is linear, then after a change of basis, $ (x_{2},x_{1}) $ is the principle pair $ (g(X),h(X)) $.  Let $ v_{i}(X,t) $ be the polynomial $ \beta^{\sharp}(x_{i}) $ and $ f_{i}(X) $ the rational function below:
    \begin{align*}
        f_{i}(X) &= v_{i}(X,-x_{2}/x_{1}) \\
        &= \beta^{\sharp}(x_{i}) \mid_{t=-x_{2}/x_{1}}.
    \end{align*}
    Because $ (x_{2},x_{1}) $ is a principle pair,
    \begin{align*}
        v_{2}(X,t) &= \beta^{\sharp}(x_{2}) \\
        &= x_{2}+tx_{1}.
    \end{align*}
    Since $ v_{2}(X,t) $ is equal to $ x_{2}+tx_{1} $,
    \begin{align*}
        f_{2}(X) &= v_{2}(X,-x_{2}/x_{1}) \\
        &= x_{2}+\left(\frac{-x_{2}}{x_{1}}\right) x_{1} \\
        &= 0.
    \end{align*}
    The ring $ k[X] $ is a UFD and the only units in $ k[X]_{x_{1}} $ are constants and powers of $ x_{1} $.  So $ f_{i}(X) $ may be uniquely represented as $ r_{i}(X)/x_{1}^{d_{i}} $, for a polynomial $ r_{i}(X) $ which is indivisible by $ x_{1} $ and a non-negative integer $ d_{i} $.  Since $ f_{i}(X) \in k[X]_{x_{1}}^{\mathbb{G}_{a}} $ and $ \mathbb{G}_{a} $ has no characters, $ r_{i}(X) $ is invariant.

    If we note that
    \begin{align*}
        f_{i}(X) \mid_{x_{2}=0} &= v_{i}(X,-x_{2}/x_{1}) \mid_{x_{2}=0} \\
        &= v_{i}(X,0) \\
        &=\beta^{\sharp}(x_{i}) \mid_{t=0} \\
        &= x_{i},
    \end{align*}
    then since $ r_{i}(X) $ is equal to $ f_{i}(X) x_{1}^{d_{i}} $,
    \begin{align*}
        r_{i}(X) \mid_{x_{2}=0} &= \left(f_{i}(X)x_{1}^{d_{i}}\right)\mid_{x_{2}=0} \\
        &= (f_{i}(X) \mid_{x_{2}=0}) x_{1}^{d_{i}} \\
        &= x_{i}x_{1}^{d_{i}}.
    \end{align*}

    We claim that $ x_{1},r_{3}(X),\dots,r_{n}(X) $ are algebraically independent.

    Suppose there is a polynomial $ g(y_{1},\dots,y_{n-1}) $ such that
    \begin{equation*}
        g(x_{1},r_{3}(X),\dots,r_{n}(X)) =0.
    \end{equation*}
    The following computations would imply that $ \{x_{1},x_{3},x_{4},\dots,x_{n}\} $ are not algebraically independent,
    \begin{align*}
        0 &= g(x_{1},r_{3}(X),\dots,r_{n}(X)) \\
        &= g(x_{1},r_{3}(X),\dots,r_{n}(X)) \mid_{x_{2} =0} \\
        &= g(x_{1},x_{3}x_{1}^{d_{3}},\dots,x_{n}x_{1}^{d_{n}}).
    \end{align*}
    However, $ \{x_{1},x_{3},\dots,x_{n}\} $ are algebraically independent.  Therefore the polynomials below:
    \begin{equation*}
        x_{1},r_{3}(X),\dots,r_{n}(X)
    \end{equation*}
    are algebraically independent as well.

    Since $ \{x_{1},r_{3}(X),r_{4}(X),\dots,r_{n}(X)\} $ are algebraically independent, and
    \begin{equation*}
        k[X]^{\mathbb{G}_{a}}_{x_{1}} = k[x_{1},r_{3}(X),\dots,r_{n}(X)]_{x_{1}},
    \end{equation*}
    the ring $ k[X]^{\mathbb{G}_{a}}_{x_{1}} $ is the localization of a polynomial ring in $ n-1 $ variables with respect to one of the indeterminates by Proposition ~\ref{P:itsTheRing} (see page \pageref{P:itsTheRing}).
\end{proof}

\section{A Beginning Look at Invariant Rings of $ \mathbb{G}_{a} $-Representations Which Have Linear Pairs.} \label{S:TA}
\begin{lem} \label{L:multiGrading}
    If $ \beta: \mathbb{G}_{a} \to \operatorname{GL}(\mathbf{V}) $ is an $ n $-dimensional, linear, representation such that $ \mathbf{V} $ decomposes into $ \ell $, indecomposable, $ \mathbb{G}_{a} $ representations, then there is an action $ \gamma $ of $ \mathbb{G}_{m}^{\ell} $ on $ \mathbf{V} $ such that $ \gamma $ and $ \beta $ commute.
\end{lem}
\begin{proof}
    Let $ \mathbf{V} \cong \oplus_{i=1}^{\ell} \mathbf{V}_{i} $ where $ \dim_{k}(\mathbf{V}_{i}) =n_{i} $.  If $ s \in \mathbb{G}_{m} $ and $ (a_{1},\dots,a_{n_{i}}) \in \mathbf{V}_{i} $, then let $ \mathbb{G}_{m} $ act on $ \mathbf{V}_{i} $ by sending $ (s,(a_{1},\dots,a_{n_{i}})) $ to $ (sa_{1},\dots,sa_{n_{i}}) $.  This action commutes with the action of $ \mathbb{G}_{a} $ on $ \mathbf{V}_{i} $.  The product of these actions is an action of $ \mathbb{G}_{m}^{\ell} $ on $ \mathbf{V} \cong \oplus_{i=1}^{\ell} \mathbf{V}_{i} $, which commutes with the action of $ \mathbb{G}_{a} $ on $ \mathbf{V} $.
\end{proof}
\begin{rmk}
    Since the action of the torus described above commutes with the action of $ \mathbb{G}_{a} $, every element of $ k[X]^{\mathbb{G}_{a}} $ is the sum of polynomials which are homogeneous with respect to the multi-grading of Lemma ~\ref{L:multiGrading} (see page \pageref{L:multiGrading}).
\end{rmk}
Let $ \beta: \mathbb{G}_{a} \to \operatorname{GL}(\mathbf{V}) $ be an $ n $-dimensional, linear representation, and let $ \{x_{1},\dots,x_{n}\} $ be a basis of $ \mathbf{V}^{\ast} $.  If there is a non-trivial, principle pair composed of two linear forms, then we may apply a change of basis and assume that $ (x_{2},x_{1}) $ is a principle pair.  Assume now that $ (x_{2},x_{1}) $ is a principle pair.  If $ \beta^{\sharp}(x_{i}) $ is equal to $ v_{i}(X,t) $, and $ f_{i}(X) $ is equal to $ v_{i}(X,-x_{2}/x_{1}) $, then by Theorem ~\ref{T:qJustification} (see page \pageref{T:qJustification}),
\begin{equation*}
    k[X]_{x_{1}}^{\mathbb{G}_{a}} = k[x_{1}^{\pm 1},f_{3}(X),\dots,f_{n}(X)].
\end{equation*}
Since $ k[X]_{x_{1}} $ is a UFD, there are polynomials $ r_{3}(X),\dots,r_{n}(X) \in k[X] \setminus \langle x_{1} \rangle k[X] $ and non-negative integers $ c_{3},\dots,c_{n} $ such that $ f_{i}(X) $ is equal to $ r_{i}(X)/x_{1}^{c_{i}} $.  Because $ \mathbb{G}_{a} $ has no characters, $ r_{i}(X) \in k[X]^{\mathbb{G}_{a}} $ for $ 3 \le i \le n $.  Also, if $ A $ is the ring $ k[r_{3}(X),\dots,r_{n}(X)] $, then $ A $ is an $ n-2 $-dimensional, polynomial ring by Theorem ~\ref{T:qJustification} (see page \pageref{T:qJustification}).  Moreover, $ x_{1} $ is algebraically independent over $ \operatorname{Frac}(A) $.
\begin{thm} \label{T:CoxRingStart}
    Let $ \beta: \mathbb{G}_{a} \to \operatorname{GL}(\mathbf{V}) $ be an $ n $-dimensional, linear representation of $ \mathbb{G}_{a} $, and let $ \{x_{1},\dots,x_{n}\} $ be a basis of $ \mathbf{V}^{\ast} $.  If $ (x_{2},x_{1}) $ is a principle pair, then let $ A $ be the $ n-2 $-dimensional, polynomial ring described above.  If $ B $ is the ring $ \left(\oplus_{i=0}^{\infty} (A: x_{1}^{i} )_{k[X]}\right) $, then
    \begin{align*}
        B &= \left(A[1/x_{1}]\right)\cap k[X] \\
        & =\left(\oplus_{i=0}^{\infty} (A: x_{1}^{i} )_{k[X]}\right) \\
        & = \left(\oplus_{i=0}^{\infty} \oplus_{L \in \mathbb{Z}^{\ell}} \left((A: x_{1}^{i})_{k[X]}\right)_{L}\right) \\
        & \cong \left(\oplus_{i=0}^{\infty} \left(\langle x_{1} \rangle  \cap A \right)^{(i)}\right) \\
        & \cong \left(\oplus_{L \in \mathbb{Z}^{\ell}} \oplus_{i=0}^{\infty} \left(\langle x_{1} \rangle \cap A \right)^{(i)}_{L}\right),
    \end{align*}
    where $ ((A:x_{1}^{i})_{k[X]})_{L} $ (respectively $ (\langle x_{1} \rangle \cap A)_{L}^{(i)}) $) denotes the multi-degree $ L $ part of $ (A:x_{1}^{i})_{k[X]} $ (respectively $ (\langle x_{1} \rangle \cap A)^{(i)} $) with respect to the multi-grading from Lemma ~\ref{L:multiGrading} (see page \pageref{L:multiGrading}).  The ring $ B $ is finitely generated if and only if $ k[X]^{\mathbb{G}_{a}} $ is as well.  Moreover
    \begin{align}
        k[X]^{\mathbb{G}_{a}} &= k[X]_{x_{1}}^{\mathbb{G}_{a}} \cap k[X], \label{E:25} \\
        &= B[x_{1}], \notag \\
        &= \left(\oplus_{i=0}^{\infty} (A: \langle x_{1}^{i} \rangle )_{k[X]}\right)[x_{1}], \label{E:26} \\
        &= \left(\oplus_{i=0}^{\infty} \oplus_{L \in \mathbb{Z}^{\ell}} \left((A: \langle x_{1}^{i}  \rangle)_{k[X]}\right)_{L}\right)[x_{1}], \label{E:27} \\
        &\cong \left(\oplus_{i=0}^{\infty} \left(\langle x_{1} \rangle  \cap A \right)^{(i)}\right)[x_{1}], \label{E:32} \\
        &\cong \left(\oplus_{i=0}^{\infty} \oplus_{L \in \mathbb{Z}^{\ell}} \left(\langle x_{1} \rangle \cap A \right)^{(i)}_{L}\right)[x_{1}]. \label{E:43}
    \end{align}
    If $ B $ is the Cox ring of a projective, $ \mathbb{Q} $-factorial, normal variety $ Z $ whose class group is free, then either $ k[X]^{\mathbb{G}_{a}} $ is a pure, degree one, transcendental extension of $ A $, or $ B $ is isomorphic to $ k[X]^{\mathbb{G}_{a}} $.
\end{thm}
\begin{proof}
    Once we prove ~\eqref{E:25}, ~\eqref{E:26}, ~\eqref{E:27}, ~\eqref{E:32}, and ~\eqref{E:43}, the preceding identities about $ B $ follow.  If $ f(X) \in k[X]^{\mathbb{G}_{a}} $, then $ f(X) \in k[X]^{\mathbb{G}_{a}}_{x_{1}} $ and $ f(X) $ is a polynomial.  Equation ~\eqref{E:25} follows.  By Proposition ~\ref{P:itsTheRing} (see page \pageref{P:itsTheRing}), $ k[X]_{x_{1}}^{\mathbb{G}_{a}} $ is equal to $ (A[x_{1}])_{x_{1}} $.

    Let $ f(X) $ be an invariant polynomial.  Because $ f(X) \in k[X]_{x_{1}}^{\mathbb{G}_{a}} $ and $ k[X]^{\mathbb{G}_{a}}_{x_{1}} \cong \left(A[x_{1}]\right)_{x_{1}} $, there is a $ d \in \mathbb{N}_{0} $ and there are $ g_{i}(R) \in A $ for $ 0 \le i \le d $ such that:
    \begin{equation} \label{E:44}
        f(X) = \sum_{i=0}^{d} \left(g_{i}(R)/x_{1}^{i}\right).
    \end{equation}
    Since $ f(X) $ is a polynomial, $ x_{1}^{i} $ divides $ g_{i}(R) $.  Therefore $ g_{i}(R)/x_{1}^{i} \in (A: x_{1}^{i})_{k[X]} $.  If $ f(X) \in (A:x_{1}^{i})_{k[X]} $, then there is a $ g(R) \in A $ such that $ g(R) $ is equal to $ x_{1}^{i}f(X) $, i.e., $ f(X) $ is equal to $ g(R)/x_{1}^{i} $.  We will be finished proving ~\eqref{E:26} once we show that the intersection of $ (A:x_{1}^{i})_{k[X]} $ and $ (A:x_{1}^{j})_{k[X]} $ is empty whenever $ i $ and $ j $ are not equal.  If $ f(X) $ is a polynomial in the intersection of $ (A:x_{1}^{i})_{k[X]} $ and $ (A: x_{1}^{j})_{k[X]} $, then there are polynomials $ g_{i}(R), g_{j}(R) \in A $ such that
    \begin{align*}
        g_{i}(R)/x_{1}^{i} &= f(X) \\
        &= g_{j}(R)/x_{1}^{j}.
    \end{align*}
    This implies that $ x_{1}^{j} g_{i}(R)-x_{1}^{i} g_{j}(R) $ is equal to zero.  This contradicts Theorem ~\ref{T:qJustification} (see page \pageref{T:qJustification}), which states that $ A[x_{1}] $ is an $ n-1 $-dimensional polynomial ring, and finishes the proof of ~\eqref{E:26}. 
    
    Let $ \mathfrak{p} $ equal $ \langle x_{1} \rangle \cap A $ and assume that $ x_{1}^{i} $ divides $ g(R) \in A $.  Since $ x_{1}^{i} $ divides $ g(R) $, the polynomial $ g(R) \in \langle x_{1} \rangle k[X] $.  So $ g(R) $ is also in $ \langle x_{1}^{i} \rangle k[X]_{\langle x_{1} \rangle} $.  Because $ \langle x_{1}^{i} \rangle k[X]_{\langle x_{1} \rangle} \cap A_{\mathfrak{p}} $ is equal to $ \mathfrak{p}^{i} A_{\mathfrak{p}} $, the polynomial $ g(R) $ is in $ \mathfrak{p}^{i} A_{\mathfrak{p}} $.  If $ \iota : A \to A_{\mathfrak{p}} $ is the natural localization map, then since $ g(R) \in A $, it is in $ \iota^{-1}(\mathfrak{p}^{i} A_{\mathfrak{p}}) $, i.e., $ g(R) \in \mathfrak{p}^{(i)} $.  As a result, if $ g(R) \in A $ is divisible by $ x_{1}^{i} $, then it is in $ \mathfrak{p}^{(i)} $.

    We shall now show that if $ g(R) \in \mathfrak{p}^{(i)} $, then $ x_{1}^{i} $ divides $ g(R) $.  Since \linebreak $ g(R) \in \mathfrak{p}^{(i)} $, it is in $ (\langle x_{1} \rangle \cap A)^{i}A_{\mathfrak{p}} $.  If $ \phi: A_{\mathfrak{p}} \to k[X]_{\langle x_{1} \rangle} $ is the inclusion homomorphism, then $ \phi(\mathfrak{p}^{i})k[X]_{\langle x_{1} \rangle} = \langle x_{1}^{i}\rangle $.  Therefore $ g(R) \in \langle x_{1}^{i} \rangle k[X]_{\langle x_{1} \rangle} $.  Because $ k[X]_{\langle x_{1} \rangle} $ is a UFD and $ g(R) $ is a polynomial, $ g(R) \in \langle x_{1}^{i} \rangle k[X] $, i.e., $ x_{1}^{i} $ divides $ g(R) $.  So $ g(R) \in \mathfrak{p}^{(i)} $ if and only if then $ x_{1}^{i} $ divides $ g(R) \in A $.

    Define a ring homomorphism $ \Psi: \left(\oplus_{i=0}^{\infty} \mathfrak{p}^{(i)}y^{i}\right)[x_{1}] \to k[X]^{\mathbb{G}_{a}} $ as follows.  For any $ g_{i}(R) \in \mathfrak{p}^{(i)} $, the ring homomorphism $ \Psi $ maps $ g_{i}(R)y^{i} $ to $ g_{i}(R)/x_{1}^{i} $.  We may $ k[x_{1}] $-linearly extend $ \Psi $ to a ring homomorphism.  We claim that $ \Psi $ is an isomorphism.  For any $ f(X) \in k[X]^{\mathbb{G}_{a}} $ there is a non-negative integer $ d $ and a collection $ \{g_{i}(R)\}_{i=0}^{d} $ of elements of $ A $ such that ~\eqref{E:44} holds.  Now the following calculations show that $ \Psi $ is surjective,
    \begin{align*}
        \Psi\left(\sum_{i=0}^{d} g_{i}(R)y^{i}\right) &= \sum_{i=0}^{d} \Psi(g_{i}(R)y^{i}), \\
        &= \sum_{i=0}^{d} g_{i}(R)/x_{1}^{i}, \\
        &= f(X).
    \end{align*}
    If 
    \begin{align}
        0 &= \Psi\left(\sum_{i=0}^{d} \sum_{j=0}^{e_{i}} g_{i,j}(R)y^{i}x_{1}^{j} \right), \notag \\
        &= \sum_{i=0}^{d} \sum_{j=0}^{e_{i}} g_{i,j}(R)/x_{1}^{i-j}, \label{E:98}
    \end{align}
    then upon multiplying the left and right hand side of ~\eqref{E:98} by $ x_{1}^{d} $, we obtain
    \begin{align*}
        A[x_{1}] & \ni \sum_{i=0}^{d} \sum_{j=0}^{e_{i}} g_{i,j}(R)x_{1}^{d-i+j} \\
        &=0.
    \end{align*}
    Since $ x_{1} $ and $ r_{3}(X),\dots,r_{n}(X) $ are algebraically independent over $ k $, each polynomial $ g_{i,j}(R) $ is equal to zero.  As a result, $ \Psi $ is an isomorphism.  Because $ \mathfrak{p} $ is equal to $ \langle x_{1} \rangle \cap A $, ~\eqref{E:32} follows.

    We may break up each element of $ (A: x_{1}^{i})_{k[X]} $ and $ (\langle x_{1} \rangle \cap A)^{(i)} $ into the sum of multi-graded parts, so ~\eqref{E:27} and ~\eqref{E:43} follow.

    Suppose there is a polynomial $ f(X) $ in $ B \setminus A $.  If this is the case, then there is a $ d \in \mathbb{N}_{0} $ and $ g_{i}(R) \in A $ for $ 0 \le i \le d $ such that $ f(X) $ is equal to $ \sum_{i=0}^{d} g_{i}(R)/x_{1}^{i} $.  So, there is an invariant $ f_{i}(X) \in B $ equal to $ g_{i}(R)/x_{1}^{i} $.  If $ \widetilde{B} $ denotes the integral closure of $ B $, then $ x_{1} \in \widetilde{B_{f_{i}(X)}} $ because $ x_{1}^{i}-g_{i}(R)/f_{i}(X) $ is equal to zero.  Since $ x_{1} $ is the only element of $ k[X]_{f_{i}(X)}^{\mathbb{G}_{a}} $ which is not in $ B_{f_{i}(X)} $,
    \begin{align*}
        k[X]_{f_{i}(X)}^{\mathbb{G}_{a}} & \cong \widetilde{B_{f_{i}(X)}}, \\
        &\cong \widetilde{B}_{f_{i}(X)}.
    \end{align*}
    Because
    \begin{align*}
        k[X]_{x_{1}f_{i}(X)}^{\mathbb{G}_{a}} & \cong k[x_{1},r_{3}(X),\dots,r_{n}(X)]_{x_{1}f_{i}(X)} \\
        & \cong \widetilde{B_{x_{1}f_{i}(X)}},
    \end{align*}
    the field of fractions of both $ \widetilde{B} $ and $ k[X]^{\mathbb{G}_{a}} $ are the same.  Since $ \widetilde{B} $ is contained in the ring of invariants, and the rings $ \widetilde{B} $ and $ k[X]^{\mathbb{G}_{a}} $ are integrally closed rings with the same field of fractions, $ k[X]^{\mathbb{G}_{a}} = \widetilde{B} $.
    
    As a result, $ k[X]^{\mathbb{G}_{a}} $ is integral over $ B $.  Therefore, $ B $ is a finitely generated $ k $ algebra if and only if $ k[X]^{\mathbb{G}_{a}} $ is as well.  If $ Z $ is a projective, $ \mathbb{Q} $-factorial, normal variety whose class group is finitely generated and free and $ \operatorname{Cox}(Z) \cong B $, then by \cite[Chapter 1, Basic Concepts, Section 4, Cox sheaves and Cox rings, Proposition 1.4.1.5]{ADHA}
    \begin{align*}
        B & \cong \operatorname{Cox}(Z) \\
        & \cong k[X]^{\mathbb{G}_{a}}.
    \end{align*}

    If there are no polynomials in $ B \setminus A $, then $ x_{1} $ is transcendental over $ B $ and $ B \cong A $.  If $ x_{1} $ is transcendental over $ B $, then
    \begin{align*}
         k[X]^{\mathbb{G}_{a}} &\cong B[x_{1}], \\
         & \cong A[x_{1}].
    \end{align*}
    So both $ B $ and $ k[X]^{\mathbb{G}_{a}} $ are finitely generated.
\end{proof}
\section{Toric Constructions Related to the Counterexample.} \label{S:specificTor}
We would like to recall some notions about toric varieties for the reader.

The reader may recall (see Chapter ~\ref{S:ToricVarieties} on page \pageref{S:ToricVarieties}) that all divisors on a toric variety are equivalent to a torus invariant divisor.  The torus invariant divisors of a toric variety correspond to the rays of the fan.  If $ \rho \in \Sigma(1) $ is a ray of the fan $ \Sigma $, then we follow the conventions of \cite{CoxLittleSchenck} by denoting the corresponding torus invariant divisor by $ D_{\rho} $.

Recall that if $ M $ is the lattice of characters of a toric variety $ X_{\Sigma} $ and $ N $ is the lattice of one parameter subgroups of $ X_{\Sigma} $, then there is a pairing $ \langle \cdot, \cdot \rangle $ obtained from the standard, Euclidean, inner product.  If $ D $ is a divisor $ \sum_{\rho \in \Sigma(1)} a_{\rho} D_{\rho} $, then \cite[Chapter 4, Divisors on Toric Varieties, Section 2, Cartier Divisors on Toric Varieties, Theorem 4.2.8]{CoxLittleSchenck} says:
\begin{quote}
    Let $ X_{\Sigma} $ be the toric variety of the fan $ \Sigma $ and let $ D = \sum_{\rho} a_{\rho} D_{\rho} $.  Then the following are equivalent:
    \begin{itemize}
        \item[a)] $ D $ is Cartier.
        \item[b)] $ D $ is principle on the affine open subset $ U_{\sigma} $ for all $ \sigma \in \Sigma $.
        \item[c)] For each $ \sigma \in \Sigma $, there is $ m_{\sigma} \in M $ with $ \langle m_{\sigma}, u_{\rho} \rangle = -a_{\rho} $ for all $ \rho \in \sigma(1) $.
        \item[d)] For each $ \sigma \in \Sigma_{\max} $, there is $ m_{\sigma} $ with $ \langle m_{\sigma}, u_{\rho} \rangle = -a_{\rho} $ for all $ \rho \in \sigma(1) $.
    \end{itemize}
\end{quote}
The collection $ m_{\sigma} \in M $ determine the support function $ \phi_{D} $ via $ \phi_{D}(u) = \langle m_{\sigma},u \rangle $ if $ u \in \sigma $ (see Chapter ~\ref{S:ToricVarieties} on page \pageref{S:ToricVarieties} regarding support functions of a Cartier divisor on a toric variety).  A toric variety is simplicial if and only if it is $ \mathbb{Q} $-factorial.  In this case, if $ D $ is a Weil divisor equal to $ \sum_{\rho} a_{\rho} D_{\rho} $, then there is a collection $ m_{\sigma} \in M_{\mathbb{Q}} $, which determine the support function $ \phi_{D} $ via $ \phi_{D}(u) = \langle m_{\sigma}, u \rangle $ if $ u \in \sigma $.

We will use a specific toric variety to create a counterexample to the \linebreak Weitzenb\"{o}ck conjecture.  The next proposition describes this toric variety, its class group, Nef cone, and intersection pairing.
\begin{prop} \label{P:specificToricVariety}
    Let $ p \in \operatorname{Spec}(\mathbb{Z}) $ and let $ j_{1},j_{2},j_{3} $ be natural numbers such that
    \begin{align*}
        j_{3} & > j_{2} >j_{1}, \\
        j_{3}+j_{2} & \equiv 1 \mod{2}.
    \end{align*}
    Let $ u_{\rho_{1}},\dots,u_{\rho_{4}} $ be the following points in $ \mathbb{R}^{2} $:
    \begin{align*}
        u_{\rho_{1}} &= (p^{j_{2}-j_{1}}-1, p^{j_{3}-j_{1}}-1) \\
        u_{\rho_{2}} &= (-p^{j_{2}-j_{1}},-p^{j_{3}-j_{1}}) \\
        u_{\rho_{3}} &= (1,0) \\
        u_{\rho_{4}} &= (0,1),
    \end{align*}
    and let $ \rho_{i} $ be the ray $ \mathbb{R}_{+} u_{\rho_{i}} $.  If $ M $ and $ N $ are both isomorphic to $ \mathbb{Z}^{2} $, and $ \Sigma $ is the fan obtained from the two dimensional cones below and their faces:
    \begin{align*}
        \sigma_{1} &= \operatorname{Cone}(u_{\rho_{1}}, u_{\rho_{4}}) \\
        \sigma_{2} &= \operatorname{Cone}(u_{\rho_{1}}, u_{\rho_{3}}) \\
        \sigma_{3} &= \operatorname{Cone}(u_{\rho_{2}}, u_{\rho_{3}}) \\
        \sigma_{4} &= \operatorname{Cone}(u_{\rho_{2}}, u_{\rho_{4}}),
    \end{align*}
    then the fan $ \Sigma $ determines a projective, simplicial, toric variety $ X_{\Sigma} $.  The class group of $ X_{\Sigma} $ is generated by the divisors $ D_{\rho_{1}} $ and $ D_{\rho_{2}} $.  Let $ D $ equal $ aD_{\rho_{1}}+bD_{\rho_{2}} $.  The support function $ \phi_{D} $ is obtained from the following points of $ M_{\mathbb{Q}} $:
    \begin{align}
        m_{\sigma_{1}} & = (-a/(p^{j_{2}-j_{1}}-1),0), \notag \\
        \sigma_{1} &= \operatorname{Cone}(u_{\rho_{1}},u_{\rho_{4}}), \notag \\
        m_{\sigma_{2}} &= (0,-a/(p^{j_{3}-j_{1}}-1)), \notag \\
        \sigma_{2} &= \operatorname{Cone}(u_{\rho_{1}},u_{\rho_{3}}), \notag \\
        m_{\sigma_{3}} &= (0,b/p^{j_{3}-j_{1}}), \notag \\
        \sigma_{3} &= \operatorname{Cone}(u_{\rho_{2}},u_{\rho_{3}}), \notag \\
        m_{\sigma_{4}} &= (b/p^{j_{2}-j_{1}},0), \notag \\
        \sigma_{4} &= \operatorname{Cone}(u_{\rho_{2}},u_{\rho_{4}}). \label{E:39}
    \end{align}
    The Nef cone of $ X_{\Sigma} $ is composed of all divisors $ aD_{\rho_{1}}+bD_{\rho_{2}} $ such that:
    \begin{align}
        a &\le 0, \notag \\
        (p^{j_{2}-j_{1}}-1)b & \ge -p^{j_{2}-j_{1}}a, \label{E:60}
    \end{align}
    and the intersection pairing on $ X_{\Sigma} $ is determined by the following rules:
    \begin{align}
        D_{\rho_{2}}^{2} &= p^{-(j_{3}+j_{2}-2j_{1})}, \notag \\
        D_{\rho_{1}} \cdot D_{\rho_{2}} &= 0, \notag \\
        D_{\rho_{1}}^{2} &= -1/(p^{j_{3}-j_{1}}-1)(p^{j_{2}-j_{1}}-1). \label{E:56}
    \end{align}
\end{prop}
\begin{proof}
    The support of $ \Sigma $ is $ N_{\mathbb{R}} \cong \mathbb{R}^{2} $.  So by Proposition ~\ref{P:toricComplete} (see page \pageref{P:toricComplete}) the variety $ X_{\Sigma} $ is complete.  Moreover, a toric surface is always simplicial.  By \cite[Chapter 6, Line Bundles on a Toric Variety, Section 3, The Nef and Mori Cones, 6.3.25]{CoxLittleSchenck}, $ X_{\Sigma} $ is projective.

    Since characters are principle divisors,
    \begin{align*}
        0 &\sim \operatorname{div}(\chi^{(1,0)}) \\
        & \sim (p^{j_{2}-j_{1}}-1)D_{\rho_{1}}-p^{j_{2}-j_{1}}D_{\rho_{2}}+D_{\rho_{3}} \\
        0 & \sim \operatorname{div}(\chi^{(0,1)}) \\
        & \sim (p^{j_{3}-j_{1}}-1)D_{\rho_{1}}-p^{j_{3}-j_{1}}D_{\rho_{2}}+D_{\rho_{4}}.
    \end{align*}
    Therefore,
    \begin{align*}
        D_{\rho_{3}} & \sim -(p^{j_{2}-j_{1}}-1)D_{\rho_{1}}+p^{j_{2}-j_{1}}D_{\rho_{2}} \\
        D_{\rho_{4}} & \sim -(p^{j_{3}-j_{1}}-1)D_{\rho_{1}}+p^{j_{3}-j_{1}}D_{\rho_{2}}.
    \end{align*}
    So, $ D_{\rho_{1}} $ and $ D_{\rho_{2}} $ generate the class group of $ X_{\Sigma} $.

    If $ D $ is the divisor $ aD_{\rho_{1}}+bD_{\rho_{2}} $, and $ m_{\sigma_{1}},\dots,m_{\sigma_{4}} $ are the points in ~\eqref{E:39}, then:
    \begin{align*}
        -a &= \langle (-a/(p^{j_{2}-j_{1}}-1),0),(p^{j_{2}-j_{1}}-1,p^{j_{3}-j_{1}}-1) \rangle \\
        &= \langle m_{\sigma_{1}}, u_{\rho_{1}} \rangle \\
        0 &= \langle (-a/(p^{j_{2}-j_{1}}-1),0), (0,1) \rangle \\
        &= \langle m_{\sigma_{1}}, u_{\rho_{4}} \rangle,
    \end{align*}
    so $ m_{\sigma_{1}} $ is indeed the point in ~\eqref{E:39}.  The following identities hold for $ m_{\sigma_{2}} $:
    \begin{align*}
        -a &= \langle (0,-a/(p^{j_{3}-j_{1}}-1)),(p^{j_{2}-j_{1}}-1, p^{j_{3}-j_{1}}-1) \rangle \\
        &= \langle m_{\sigma_{2}}, u_{\rho_{1}} \rangle \\
        0 &= \langle (0,-a/(p^{j_{3}-j_{1}}-1)), (1,0) \rangle \\
        &= \langle m_{\sigma_{2}}, u_{\rho_{3}} \rangle,
    \end{align*}
    the next identities hold for $ m_{\sigma_{3}} $:
    \begin{align*}
        -b &= \langle (0,b/p^{j_{3}-j_{1}}),(-p^{j_{2}-j_{1}}, -p^{j_{3}-j_{1}}) \rangle \\
        &= \langle m_{\sigma_{3}}, u_{\rho_{2}} \rangle \\
        0 &= \langle (0,b/p^{j_{3}-j_{1}}), (1,0) \rangle \\
        &= \langle m_{\sigma_{3}}, u_{\rho_{3}} \rangle,
    \end{align*}
    and finally, the following identities hold for $ m_{\sigma_{4}} $:
    \begin{align*}
        -b &= \langle (b/p^{j_{2}-j_{1}},0), (-p^{j_{2}-j_{1}}, -p^{j_{3}-j_{1}}) \rangle \\
        &= \langle m_{\sigma_{4}}, u_{\rho_{2}} \rangle \\
        0 &= \langle (b/p^{j_{2}-j_{1}},0), (0,1) \rangle \\
        &= \langle m_{\sigma_{4}}, u_{\rho_{4}} \rangle.
    \end{align*}
    Therefore $ m_{\sigma_{1}},\dots,m_{\sigma_{4}} $ are the points described in ~\eqref{E:39}.

    Cox, Little and Schenck give a criterion for a divisor on a projective, simplicial, toric variety to be Nef in \cite[Chapter 6, Line Bundles on Toric Varieties, Section 4, The Simplicial Case, Thoerem 6.4.9]{CoxLittleSchenck}:
    \begin{quote}
        Let $ X_{\Sigma} $ be a projective simplicial toric variety.  Then:
        \begin{itemize}
            \item[a)] A Cartier divisor $ D $ is nef if and only if its support function $ \phi_{D} $ satisfies:
            \begin{equation*}
                \phi_{D}(u_{\rho_{1}}+\cdots+u_{\rho_{k}}) \ge \phi_{D}(u_{\rho_{1}})+\cdots + \phi_{D}(u_{\rho_{k}}),
            \end{equation*}
            for all primitive collections $ P = \{\rho_{1},\dots,\rho_{k}\} $ of $ \Sigma $.
            \item[b)] A Cartier divisor $ D $ is ample if and only if its support function $ \phi_{D} $ satisfies:
            \begin{equation*}
                \phi_{D}(u_{\rho_{1}}+\cdots+u_{\rho_{k}}) \ge \phi_{D}(u_{\rho_{1}})+\cdots+\phi_{D}(u_{\rho_{k}}),
            \end{equation*}
            for all primitive collections $ P = \{u_{\rho_{1}},\dots,u_{\rho_{k}}\} $ of $ \Sigma $.
        \end{itemize}
    \end{quote}

    The primitive collections of $ \Sigma $ are $ \{\rho_{1},\rho_{2}\} $ and $ \{\rho_{3},\rho_{4}\} $.  Observe that \linebreak $ u_{\rho_{1}}+u_{\rho_{2}} $ is equal to $ (-1,-1) $ and that $ u_{\rho_{3}}+u_{\rho_{4}} $ is equal to $ (1,1) $.  The following inequalities show that $ (-1,-1) \in \sigma_{4} $ and $ (1,1) \in \sigma_{2} $:
    \begin{align*}
        (-1,-1) &= \frac{1}{p^{j_{2}-j_{1}}} (-p^{j_{2}-j_{1}}, -p^{j_{3}-j_{1}}) +\frac{p^{j_{3}-j_{1}}-p^{j_{2}-j_{1}}}{p^{j_{2}-j_{1}}}(0,1) \\
        &=\frac{1}{p^{j_{2}-j_{1}}} u_{\rho_{2}}+ \frac{p^{j_{3}-j_{1}}-p^{j_{2}-j_{1}}}{p^{j_{2}-j_{1}}} u_{\rho_{4}} \\
        &\in \mathbb{R}_{+}u_{\rho_{2}}+\mathbb{R}_{+}u_{\rho_{4}} \\
        &= \operatorname{Cone}(u_{\rho_{2}}, u_{\rho_{4}}) \\
        &=\sigma_{4} \\
        (1,1) &= \frac{1}{p^{j_{3}-j_{1}}-1}(p^{j_{2}-j_{1}}-1,p^{j_{3}-j_{1}}-1) +\frac{p^{j_{3}-j_{1}}-p^{j_{2}-j_{1}}}{p^{j_{3}-j_{1}}-1} (1,0) \\
        &= \frac{1}{p^{j_{3}-j_{1}}-1} u_{\rho_{1}}+\frac{p^{j_{3}-j_{1}}-p^{j_{2}-j_{1}}}{p^{j_{3}-j_{1}}-1} u_{\rho_{3}} \\
        &\in \mathbb{R}_{+} u_{\rho_{1}}+\mathbb{R}_{+}u_{\rho_{3}}  \\
        &= \operatorname{Cone}(u_{\rho_{1}}, u_{\rho_{3}}) \\
        &= \sigma_{2}.
    \end{align*}
    The earlier criterion and ~\eqref{E:39} show that the following inequalities determine which divisors $ aD_{\rho_{1}}+bD_{\rho_{2}} $ are Nef,
    \begin{align*}
        -b/p^{j_{2}-j_{1}} &=\langle (b/p^{j_{2}-j_{1}},0),(-1,-1) \rangle \\
        &= \langle m_{\sigma_{4}}, (-1,-1) \rangle \\
        &= \phi_{aD_{\rho_{1}}+bD_{\rho_{2}}}((-1,-1)) \\
        &= \phi_{aD_{\rho_{1}}+bD_{\rho_{2}}}(u_{\rho_{1}}+u_{\rho_{2}}) \\
        & \ge \phi_{aD_{\rho_{1}}+bD_{\rho_{2}}}(u_{\rho_{1}})+\phi_{aD_{\rho_{1}}+bD_{\rho_{2}}}(u_{\rho_{2}}) \\
        &= -a-b \\
        -a/(p^{j_{3}-j_{1}}-1) &= \langle (0,-a/(p^{j_{3}-j_{1}}-1)), (1,1) \rangle \\
        &= \langle m_{\sigma_{2}}, (1,1) \rangle \\
        &= \phi_{aD_{\rho_{1}}+bD_{\rho_{2}}}((1,1)) \\
        &= \phi_{aD_{\rho_{1}}+bD_{\rho_{2}}}(u_{\rho_{3}}+u_{\rho_{4}}) \\
        &\ge \phi_{aD_{\rho_{1}}+bD_{\rho_{2}}}(u_{\rho_{3}}) + \phi_{aD_{\rho_{1}}+bD_{\rho_{2}}}(u_{\rho_{4}}) \\
        &= 0.
    \end{align*}
    These inequalities simplify to the ones in ~\eqref{E:60}.

    For a divisor $ D $ equal to $ \sum_{\rho \in \Sigma(1)} a_{\rho}D_{\rho} $ with full dimensional lattice polytope $ P_{D} $ described below:
    \begin{equation*}
        P_{D} = \{m \in M_{\mathbb{R}}, \, \, \text{s.t.} \, \, \langle m,u_{\rho} \rangle \ge -a_{\rho} \quad \rho \in \Sigma(1)\},
    \end{equation*}
    the self intersection $ D^{2} $ is equal to twice the area of $ P_{D} $ by \cite[Chapter 9, Sheaf Cohomology of Toric Varieties, Section 4, Lattice Polytopes and Differential Forms, Example 9.4.4]{CoxLittleSchenck}.  If $ m \in M_{\mathbb{R}} $ is equal to $ (m_{1},m_{2}) $,then the polytope for $ D_{\rho_{2}} $ is defined by the inequalities below:
    \begin{align*}
        (p^{j_{2}-j_{1}}-1)m_{1}+(p^{j_{3}-j_{1}}-1)m_{2} &= \langle (m_{1},m_{2}), u_{\rho_{1}} \rangle \\
        &\ge 0 \\
        -p^{j_{2}-j_{1}}m_{1}-p^{j_{3}-j_{1}}m_{2} &= \langle (m_{1},m_{2}), u_{\rho_{2}} \rangle \\
        &\ge -1 \\
        m_{1} & = \langle (m_{1},m_{2}), u_{\rho_{3}} \rangle \\
        & \ge 0 \\
        m_{2} &= \langle (m_{1},m_{2}), u_{\rho_{4}}, \rangle \\
        & \ge 0.
    \end{align*}
    The first quadrant is properly contained in the set of points
    \begin{equation*}
        m_{2} \ge -(p^{j_{2}-j_{1}}-1)/(p^{j_{3}-j_{1}}-1)m_{1}.
    \end{equation*}
    The only portion of the first quadrant contained within the set of points \linebreak $ m_{2} \le -p^{-(j_{3}-j_{2})}m_{1}+p^{-(j_{3}-j_{1})} $ is the triangle with vertices $ (0,0),(0,p^{-(j_{3}-j_{1})}) $ and $ (p^{-(j_{2}-j_{1})},0) $.  Therefore, the following identities hold:
    \begin{align*}
        D_{\rho_{2}}^{2} & = 2 \operatorname{area}(P_{D_{\rho_{2}}}) \\
        &= p^{-(j_{3}+j_{2}-2j_{1})}.
    \end{align*}
    There is a wall relation described below:
    \begin{equation*}
        p^{j_{3}-j_{1}}u_{\rho_{1}}+(p^{j_{3}-j_{1}}-1)u_{\rho_{2}}- ((p^{j_{2}-j_{1}}-1)p^{j_{3}-j_{1}}-(p^{j_{3}-j_{1}}-1)p^{j_{2}-j_{1}})u_{\rho_{3}} =(0,0).
    \end{equation*}
    By \cite[Chapter 6, Line Bundles on Projective Toric Varieties, Section 4, The Simplicial Case, Proposition 6.4.4]{CoxLittleSchenck} (see Chapter ~\ref{S:ToricVarieties} on page \pageref{S:ToricVarieties} regarding the definition of the multiplicity of a cone $ \sigma $)
    \begin{align*}
        D_{\rho_{2}} \cdot D_{\rho_{3}} &= \operatorname{mult}(\rho_{3})/\operatorname{mult}(\sigma_{3}) \\
        D_{\rho_{1}} \cdot D_{\rho_{3}} &= \operatorname{mult}(\rho_{3})/\operatorname{mult}(\sigma_{2}).
    \end{align*}
    The multiplicities of $ \rho_{3}, \sigma_{2} $ and $ \sigma_{3} $ are:
    \begin{align*}
        \operatorname{mult}(\rho_{3}) &= 1 \\
        \operatorname{mult}(\sigma_{3}) &= p^{j_{3}-j_{1}} \\
        \operatorname{mult}(\sigma_{2}) &= p^{j_{3}-j_{1}}-1.
    \end{align*}
    As a result,
    \begin{align}
        p^{-(j_{3}-j_{1})} &= \operatorname{mult}(\rho_{3})/\operatorname{mult}(\sigma_{3}), \notag \\
        &= D_{\rho_{2}} \cdot D_{\rho_{3}}, \notag \\
        &= -(p^{j_{2}-j_{1}}-1)D_{\rho_{1}} \cdot D_{\rho_{2}}+p^{j_{2}-j_{1}} D_{\rho_{2}}^{2}, \notag\\
        &= -(p^{j_{2}-j_{1}}-1)D_{\rho_{1}} \cdot D_{\rho_{2}}+p^{-(j_{3}-j_{1})}, \label{E:40} \\
        1/(p^{j_{3}-j_{1}}-1) &= \operatorname{mult}(\rho_{3})/\operatorname{mult}(\sigma_{2}), \notag \\
        &= D_{\rho_{1}} \cdot D_{\rho_{3}}, \notag \\
        &= -(p^{j_{2}-j_{1}}-1)D_{\rho_{1}}^{2}+p^{j_{2}-j_{1}} D_{\rho_{1}} \cdot D_{\rho_{2}}. \label{E:45}
    \end{align}
    After using ~\eqref{E:40} to solve for $ D_{\rho_{1}} \cdot D_{\rho_{2}} $ and then ~\eqref{E:45} to solve for $ D_{\rho_{1}}^{2} $ we obtain the formulae in ~\eqref{E:56}.
\end{proof}
Let us recall the construction of a simplicial, toric, variety as the quotient of a quasi affine variety by a diagonalizable group.

Let $ G $ equal $ \operatorname{Hom}_{\mathbb{Z}}(\operatorname{Cl}(X_{\Sigma}), \mathbb{G}_{m}) $.  The points of the group $ G $ are the points $ (a_{1},\dots,a_{\operatorname{card}(\Sigma(1))}) $ such that $ \prod_{i=1}^{\operatorname{card}(\Sigma(1))} a_{i}^{\langle m, u_{\rho_{i}} \rangle} $ is equal to one for all \linebreak $ m \in M $ by \cite[Chapter 5, Homogeneous Coordinates on Toric Varieties, Section 1, Quotient Constructions of Toric Varieties, Lemma 5.1.1]{CoxLittleSchenck}.  If $ X_{\Sigma} $ has no torus factors, then the following sequence is exact:
\begin{equation*}
    \xymatrix{
         0 \ar[r] & M \ar[r] & \mathbb{Z}^{\operatorname{card}(\Sigma(1))} \ar[r] & \operatorname{Cl}(X_{\Sigma}) \ar[r] & 0
    }.
\end{equation*}
Therefore, the next sequence is exact:
\begin{equation*}
    \xymatrix{
        1 \ar[r] & G \ar[r] & \mathbb{G}_{m}^{\operatorname{card}(\Sigma(1))} \ar[r] & T_{X_{\Sigma}} \ar[r] & 1
    }.
\end{equation*}
If $ \sigma $ is a strongly convex, rational, polyhedral cone of $ \Sigma $, and the affine coordinate ring of $ \mathbb{A}^{\operatorname{card}(\Sigma(1))}_{k} $ is $ k[y_{\rho_{1}},\dots,y_{\rho_{\operatorname{card}(\Sigma(1))}}] $, then let $ \widehat{\sigma} $ be the monomial $ \prod_{\rho \notin \sigma} y_{\rho} $.  The ideal $ B(\Sigma) $ is the ideal $ \langle \widehat{\sigma} \rangle_{\sigma \in \Sigma} $.  The sub-variety $ Z(\Sigma) $ is the zero set of $ B(\Sigma) $.  Another way to think of the sub-variety $ Z(\Sigma) $ is via primitive collections.

A subset $ C \subseteq \Sigma(1) $ is a primitive collection if $ C \not \subseteq \sigma(1) $ for all $ \sigma \in \Sigma $, and for every proper subset $ C^{'} \subseteq C $ there is a $ \sigma \in \Sigma $ such that $ C^{'} \subseteq \sigma(1) $.  The variety $ Z(\Sigma) $ is equal to $ \cup_{C} \mathcal{V}(\langle y_{\rho} \rangle_{\rho \in C}) $.

By \cite[Chapter 5, Section 1, Theorem 5.1.11)]{CoxLittleSchenck}, the variety $ X_{\Sigma} $ is isomorphic to $ \left( \mathbb{A}^{\operatorname{card}(\Sigma(1))}_{k} \setminus Z(\Sigma) \right) //G $, and $ \mathbb{A}^{\operatorname{card}(\Sigma(1))}_{k} \setminus Z(\Sigma) \to X_{\Sigma} $ is a geometric quotient if and only if $ \Sigma $ is simplicial.

\section{The Main Result.} \label{S:Main}
\begin{thm} \label{T:mainResult}
    If $ k $ is a field of characteristic $ p>0 $, then there exists a linear representation $ \beta: \mathbb{G}_{a} \to \operatorname{GL}(\mathbf{V}) $ such that the ring $ S_{k}(\mathbf{V}^{\ast})^{\mathbb{G}_{a}} $ is not a finitely generated $ k $-algebra.
\end{thm}
\begin{proof}
    Let $ j_{1},j_{2},j_{3} $ be natural numbers such that
    \begin{align*}
        j_{3} & > j_{2} >j_{1}, \\
        j_{3}+j_{2} & \equiv 1 \mod{2}.
    \end{align*}
    Let $ \gamma_{1},\gamma_{2},\gamma_{3} \in k^{\ast} $ be three elements such that for any $ 1 \le i<\ell \le 3 $
    \begin{equation*}
        \mathbb{F}_{p}(\sqrt[p^{j_{i}}-1]{\gamma_{i}}) \cap \mathbb{F}_{p}(\sqrt[p^{j_{\ell}}-1]{\gamma_{\ell}}) = \mathbb{F}_{p}.
    \end{equation*}
    Such elements clearly exist.  If $ a_{i}(t) $ is equal to $ t^{p^{j_{i}}}-\gamma_{i}t $, then $ \mathfrak{O}(a_{i}(t),a_{\ell}(t)) $ is equal to $ \mathfrak{O}\langle t \rangle $ for $ 1 \le i <\ell \le 3 $.  Therefore, there exist additive polynomials $ a_{1}(t),a_{2}(t) $ and $ a_{3}(t) $ such that $ \deg(a_{i}(t)) $ is equal to $ p^{j_{i}} $ and $ \mathfrak{O}(a_{i}(t),a_{\ell}(t)) = \mathfrak{O}\langle t \rangle $ for $ 1 \le i <\ell \le 3 $.  Let $ a_{1}(t),a_{2}(t),a_{3}(t) \in \mathfrak{O} $ be additive polynomials such that $ \deg(a_{i}(t)) $ is equal to $ p^{j_{i}} $ and $ \mathfrak{O}(a_{i}(t),a_{\ell}(t)) = \mathfrak{O}\langle t \rangle $ for $ 1 \le i <\ell \le 3 $.

    Let $ \mathbf{V} $ be a six dimensional vector space with dual basis $ \{x_{1},\dots,x_{6}\} $.  If $ \beta: \mathbb{G}_{a} \to \operatorname{GL}(\mathbf{V}) $ is the linear representation which acts on $ \{x_{1},\dots,x_{6}\} $ as follows:
    \begin{align*}
        x_{1} & \mapsto x_{1} \\
        x_{2} & \mapsto x_{2}+tx_{1} \\
        x_{3} & \mapsto x_{3} \\
        x_{4} & \mapsto x_{4}+a_{1}(t)x_{2} \\
        x_{5} & \mapsto x_{5}+a_{2}(t)x_{2} \\
        x_{6} & \mapsto x_{6}+a_{3}(t)x_{2},
    \end{align*}
    then we claim that $ k[X]^{\mathbb{G}_{a}} $ is not a finitely generated $ k $-algebra.  The pair $ (x_{2},x_{1}) $ is a principle pair.  As a result, if $ r_{1}(X),r_{2}(X),r_{3}(X) $ are the polynomials below:
    \begin{align}
        r_{1}(X) &= \left(x_{4}+a_{1}(-x_{2}/x_{1})x_{3}\right)x_{1}^{p^{j_{1}}}, \notag \\
        r_{2}(X) &= \left(x_{5}+a_{2}(-x_{2}/x_{1})x_{3}\right)x_{1}^{p^{j_{2}}}, \notag \\
        r_{3}(X) &= \left(x_{6}+a_{3}(-x_{2}/x_{1})x_{3}\right)x_{1}^{p^{j_{3}}}, \label{E:89}
    \end{align}
    then $ k[X]_{x_{1}}^{\mathbb{G}_{a}} = k[x_{1},x_{2},r_{1}(X),r_{2}(X),r_{3}(X)]_{x_{1}} $ by Proposition ~\ref{P:itsTheRing} (see page \pageref{P:itsTheRing}).

    The following identities hold:
    \begin{align*}
        r_{1}(X) & \equiv (-x_{2})^{p^{j_{1}}}x_{3} \mod{\langle x_{1} \rangle k[X]} \\
        r_{2}(X) & \equiv (-x_{2})^{p^{j_{2}}}x_{3} \mod{\langle x_{1} \rangle k[X]} \\
        r_{3}(X) & \equiv (-x_{2})^{p^{j_{3}}}x_{3} \mod{\langle x_{1} \rangle k[X]}.
    \end{align*}

    So, if $ A $ is the ring $ k[x_{3},r_{1}(X),r_{2}(X),r_{3}(X)] $, then the ideal $ \langle x_{1} \rangle \cap A $ is equal to
    \begin{multline*}
        \mathfrak{p}= \langle r_{1}(X)^{p^{j_{2}-j_{1}}}-x_{3}^{p^{j_{2}-j_{1}}-1}r_{2}(X),\\
        r_{1}(X)^{p^{j_{3}-j_{1}}}-x_{3}^{p^{j_{3}-j_{1}}-1}r_{3}(X),
        r_{2}(X)^{p^{j_{3}-j_{2}}}-x_{3}^{p^{j_{3}-j_{2}}-1}r_{3}(X) \rangle.
    \end{multline*}
    The ring $ A $ is a polynomial ring, so let $ y_{1},\dots,y_{4} $ be the functions defined below:
    \begin{align*}
        y_{1} &= x_{3} \\
        y_{2} &= r_{1}(X) \\
        y_{3} &= r_{2}(X) \\
        y_{4} &= r_{3}(X).
    \end{align*}
    If the affine coordinate ring of $ \mathbb{G}_{m}^{2} $ is $ k[z_{1}^{\pm 1},z_{2}^{\pm 1}] $, then let $ \mathbb{G}_{m}^{2} $ act on \linebreak $ \operatorname{Spec}(k[y_{1},\dots,y_{4}]) $ via the following co-action:
    \begin{align}
        y_{1} & \mapsto z_{2} y_{1}, \notag \\
        y_{2} & \mapsto z_{1}^{p^{j_{1}}}z_{2} y_{2}, \notag \\
        y_{3} & \mapsto z_{1}^{p^{j_{2}}}z_{2} y_{3}, \notag \\
        y_{4} & \mapsto z_{1}^{p^{j_{3}}}z_{2} y_{4}. \label{E:38}
    \end{align}
    If $ \widehat{Z} $ is the sub-variety
    \begin{equation*}
        \mathcal{V}(\langle y_{1}y_{3},y_{2}y_{3}, y_{1}y_{4}, y_{2}y_{4} \rangle) \cong \left(\mathcal{V}(\langle y_{1},y_{2} \rangle) \cup \mathcal{V}(\langle y_{3},y_{4} \rangle)\right),
    \end{equation*}
    then $ \left(\mathbb{A}^{4}_{k} \setminus \widehat{Z} \right)//\mathbb{G}_{m}^{2} $ is a variety which we claim is isomorphic to a toric variety $ X_{\Sigma} $ and that, after a suitable alteration of the grading, $ \operatorname{Cox}(X_{\Sigma}) $ is bi-graded with the bi-grading from Lemma ~\ref{L:multiGrading} (see page \pageref{L:multiGrading}).

    If $ u_{\rho_{1}},\dots,u_{\rho_{4}} $ are the following vectors in $ \mathbb{R}^{2} $:
    \begin{align*}
        u_{\rho_{1}} &= (p^{j_{2}-j_{1}}-1, p^{j_{3}-j_{1}}-1) \\
        u_{\rho_{2}} &= (-p^{j_{2}-j_{1}},-p^{j_{3}-j_{1}}) \\
        u_{\rho_{3}} &= (1,0) \\
        u_{\rho_{4}} &= (0,1),
    \end{align*}
    then the rays $ \rho_{i} = \mathbb{R}_{+} u_{\rho_{i}} $ and the faces of the two dimensional, strongly convex, rational, polyhedral cones $ \sigma_{1},\dots,\sigma_{4} $ below:
    \begin{align*}
        \sigma_{1} &= \operatorname{Cone}(u_{\rho_{1}}, u_{\rho_{4}}) \\
        \sigma_{2} &= \operatorname{Cone}(u_{\rho_{1}}, u_{\rho_{3}}) \\
        \sigma_{3} &= \operatorname{Cone}(u_{\rho_{2}}, u_{\rho_{3}}) \\
        \sigma_{4} &= \operatorname{Cone}(u_{\rho_{2}}, u_{\rho_{4}}),
    \end{align*}
    determine a fan $ \Sigma $.  This fan is the fan of the toric variety $ X_{\Sigma} $ from Proposition ~\ref{P:specificToricVariety} (see page \pageref{P:specificToricVariety}).  We claim that the toric variety $ X_{\Sigma} $ obtained from the fan $ \Sigma $ is isomorphic to $ \left(\mathbb{A}^{4}_{k} \setminus \widehat{Z}\right)//\mathbb{G}_{m}^{2} $, that
    \begin{align}
        \operatorname{bideg}(D_{\rho_{1}})&=(0,1), \notag \\
        \operatorname{bideg}(D_{\rho_{2}})&=(1,1), \notag \\
        \operatorname{bideg}(D_{\rho_{3}})&=(p^{j_{2}-j_{1}},1), \notag \\
        \operatorname{bideg}(D_{\rho_{4}})&=(p^{j_{3}-j_{2}},1), \label{E:10}
    \end{align}
    and that
    \begin{align*}
        X_{\Sigma} & \cong \left(\mathbb{A}^{4}_{k} \setminus Z(\Sigma)\right)//\mathbb{G}_{m}^{2} \\
        & =\left(\mathbb{A}^{4}_{k} \setminus \widehat{Z} \right)//\mathbb{G}_{m}^{2}.
    \end{align*}
    If $ B $ and $ C $ are the matrices below:
    \begin{align*}
        B &= \left(
        \begin{matrix}
            p^{j_{2}-j_{1}}-1 & -p^{j_{2}-j_{1}} & 1 & 0 \\
            p^{j_{3}-j_{1}}-1 & -p^{j_{3}-j_{1}} & 0 & 1
        \end{matrix}
        \right) \\
        C &= \left(
        \begin{matrix}
            0 & 1 \\
            1 & 1 \\
            p^{j_{2}-j_{1}} & 1 \\
            p^{j_{3}-j_{1}} & 1
        \end{matrix}
        \right),
    \end{align*}
    then the following sequence is exact:
    \begin{equation*}
    \xymatrix{
        0 \ar[r] & \mathbb{Z}^{2} \ar[r]^{B} & \mathbb{Z}^{4} \ar[r]^{C} & \mathbb{Z}^{2} \ar[r] & 0
        }.
    \end{equation*}
    Therefore the next sequence is exact:
    \begin{equation*}
    \xymatrix{
        0 \ar[r] & M \ar[r]^{B} & \oplus_{\rho \in \Sigma(1)} \mathbb{Z} D_{\rho} \ar[r]^{C} & \operatorname{Cl}(X_{\Sigma}) \ar[r] & 0
        }.
    \end{equation*}
    As a result, bi-grading of $ \operatorname{Cox}(X_{\Sigma}) $ is exactly that of ~\eqref{E:10}.  Also,
    \begin{align*}
        Z(\Sigma) &= \mathcal{V}(\langle y_{1},y_{2} \rangle) \cup \mathcal{V}(\langle y_{3},y_{4} \rangle) \\
        &= \widehat{Z}.
    \end{align*}
    So $ X_{\Sigma} $ is isomorphic to the variety $ \left(\mathbb{A}^{4}_{k} \setminus \widehat{Z}\right)//\mathbb{G}_{m}^{2} $ and the bi-grading of the Cox ring of $ X_{\Sigma} $ is the one from ~\eqref{E:10}.  We may endow the Cox ring of $ X_{\Sigma} $ with the bi-grading from Lemma ~\ref{L:multiGrading} (see page \pageref{L:multiGrading}) by setting an element of bi-degree $ (a,b) $ to be of bi-degree $ (p^{j_{1}}a,b) $.  Let $ \mathfrak{p} $ be the ideal below:
    \begin{equation*}
        \mathfrak{p}= \langle y_{2}^{p^{j_{2}-j_{1}}}-y_{1}^{p^{j_{2}-j_{1}}-1}y_{3},y_{2}^{p^{j_{3}-j_{1}}}-y_{1}^{p^{j_{3}-j_{1}}-1}y_{4},
        y_{3}^{p^{j_{3}-j_{2}}}-y_{1}^{p^{j_{3}-j_{2}}-1}y_{4} \rangle.
    \end{equation*}
    The variety $ \mathcal{V}(\mathfrak{p}) $ is the point $ (1,1,1,1) $ in terms of the homogeneous coordinates of $ X_{\Sigma} $ (see \cite[Chapter 5, Homogeneous Coordinates on Toric Varieties, Section 2, The Total Coordinate Ring]{CoxLittleSchenck}).  With the current isomorphism of the torus of $ X_{\Sigma} $, the point $ \mathcal{V}(\mathfrak{p}) $ is the identity $ e $ of the torus.  Therefore, $ X_{\Sigma} $ is smooth at $ \mathcal{V}(\mathfrak{p}) = e $.  Let us denote $ \operatorname{Bl}_{e}(X_{\Sigma}) $ by $ W $.  Since $ X_{\Sigma} $ is smooth at $ e $ and isomorphic to $ W $ outside the exceptional divisor, $ W $ is normal.  Since the exceptional divisor is Cartier and $ X_{\Sigma} $ is $ \mathbb{Q} $-factorial, $ W $ is also $ \mathbb{Q} $-factorial.  Therefore, the Cox ring of $ W $ is a well defined object.  By Theorem ~\ref{T:CoxRingStart} (see page \pageref{T:CoxRingStart}), the Cox ring of $ W $ is finitely generated if and only if $ k[X]^{\mathbb{G}_{a}} $ is as well.  Moreover, $ x_{1} $ is not transcendental over $ \operatorname{Cox}(W) $.  So the Cox ring of $ W $ is isomorphic to $ k[X]^{\mathbb{G}_{a}} $ by Theorem ~\ref{T:CoxRingStart}.

    For a surface, the Mori cone is equal to the pseudo-effective cone of divisors.  We claim that the Mori cone of $ W $ is not a rational polyhedral cone.  Because the Mori cone of $ W $ is the support (see Definition ~\ref{D:support} on page \pageref{D:support}) of the Cox ring of $ W $, the Cox ring of $ W $ is not a finitely generated $ k $-algebra if the Mori cone is not a rational polyhedral cone by Lemma ~\ref{L:finGenSupp} (see page \pageref{L:finGenSupp}).  If the Cox ring of $ W $ is not finitely generated as a $ k $-algebra, then $ k[X]^{\mathbb{G}_{a}} $ is not a finitely generated $ k $-algebra either by Theorem \ref{T:CoxRingStart} (see page \pageref{T:CoxRingStart}).  We claim that the Mori cone of $ W $ contains a non-rational, extremal ray.  Hence $ \overline{NE}(W) $ is not a rational, polyhedral cone.

    Suppose that $ c,d $ are natural numbers such that
    \begin{equation*}
        p^{-(j_{3}+j_{2}-2j_{1})/2}-d/c>0.
    \end{equation*}
    We aim to show that $ cD_{\rho_{2}}-dE $ is in the pseudo-effective cone for all such $ c,d $.

    Since $ j_{3}+j_{2}-2j_{1}>2 $,
    \begin{align*}
        c/d &> p^{(j_{3}+j_{2}-2j_{1})/2} \\
        &>p.
    \end{align*}
    By Proposition ~\ref{P:specificToricVariety} (see page \pageref{P:specificToricVariety})
    \begin{align*}
        (cD_{\rho_{2}}-dE)^{2}/d^{2} &= (c/d)^{2}p^{-(j_{3}+j_{2}-2j_{1})/2}-1 \\
        &>(c/d)p^{-(j_{3}+j_{2}-2j_{1})/2}-1 \\
        &>0.
    \end{align*}
    By ~\eqref{E:60} the divisor $ A_{n} $ equal to $ n D_{\rho_{2}}-D_{\rho_{1}} $ is an ample divisor of $ X_{\Sigma} $ for all natural numbers $ n $.

    For any toric variety $ X_{\Sigma} $, there is a fan $ \Sigma_{1} $, which is a refinement of $ \Sigma $, such that the natural morphism $ \pi: X_{\Sigma_{1}} \to X_{\Sigma} $ is a resolution of singularities.  If $ \pi_{1}: \operatorname{Bl}_{e}(X_{\Sigma_{1}}) \to W $ is the natural, projection morphism, then $ \pi_{1} $ is a resolution of the singularities of $ W $ because $ X_{\Sigma} $ is smooth at $ e $.  If $ E_{1} $ is the exceptional divisor of $ \pi_{1} $, and $ n_{0} $ is large enough so that $ A_{n_{0}} $ is ample, then there exist integers $ s,r $ such that $ \pi^{\ast}(A_{n_{0}})-sE-rE_{1} $ is an ample divisor of $ \operatorname{Bl}_{e}(X_{\Sigma_{1}}) $.  Note that if $ s,r $ work for $ n_{0} $, then $ s,r $ work for all $ n>n_{0} $.  This is because $ \pi^{\ast}((n-n_{0})D_{\rho_{2}}) $ is nef and the sum of a nef and ample divisor is ample.  If $ n>(sdp^{j_{3}+j_{2}-2j_{1}})/c $, then
    \begin{align*}
        \left(\pi_{1}^{\ast}(cD_{\rho_{2}}-dE)\right)^{2} &= (cD_{\rho_{2}}-dE)^{2} \\
        &>0
    \end{align*}
    and
    \begin{multline*}
        ((\pi \circ \pi_{1})^{\ast}(A_{n})-s\pi_{1}^{\ast}(E)-rE_{1})\cdot \pi_{1}^{\ast}(cD_{\rho_{2}}-dE) \\
        =(\pi^{\ast}(A_{n})-sE) \cdot (cD_{\rho_{2}}-dE) \\
        = (nD_{\rho_{2}}-D_{\rho_{1}}-sE) \cdot (cD_{\rho_{2}}-dE) \\
        = ncp^{-(j_{3}+j_{2}-2j_{1})}-sd >0.
    \end{multline*}

    By \cite[V, Surfaces, Section 1, Geometry on a Surface, Corollary 1.8]{HartshorneAG} if $ A $ is an ample divisor on a non-singular surface $ Z $, and $ D $ is a divisor such that $ D \cdot A>0 $ and $ D^{2}>0 $, then for all $ r>>0 $, $ rD $ is linearly equivalent to an effective divisor.  This result shows that $ r(\pi_{1}^{\ast}(cD_{\rho_{2}}-dE)) $ is linearly equivalent to an effective divisor for $ r>>0 $.  So $ r(cD_{\rho_{2}}-dE) $ is linearly equivalent to an effective divisor for $ r>>0 $.  Hence $ cD_{\rho_{2}}-dE $ is in the pseudo-effective cone of $ W $.

    If $ Z $ is a projective variety, $ \mathcal{L} $ is a Nef divisor on $ Z $, and $ z \in Z $ is a point, then let $ \pi: \operatorname{Bl}_{e}(Z) \to Z $ be the natural projection morphism and $ E $ the exceptional divisor on $ \operatorname{Bl}_{e}(Z) $.  The \emph{Seshadri constant} $ \epsilon(Z,\mathcal{L},z) $ is the maximum real number $ \epsilon>0 $ such that $ \pi^{\ast}(\mathcal{L})-\epsilon E $ is Nef (see \cite[Chapter 5, Local Positivity, Section 1, Seshadri Constants]{LazarsfeldI}).

    The pseudo-effective cone contains the divisor $ D_{\rho_{2}}-p^{-(j_{3}+j_{2}-2j_{1})/2}E $, because $ cD_{\rho_{2}}-dE $ is in the pseudo-effective cone for all $ c,d \in \mathbb{N} $ such that $ p^{-(j_{3}+j_{2}-2j_{1})/2}-d/c>0 $.  The Mori cone of $ W $ is equal to the pseudo-effective cone of $ W $ because $ W $ is a surface.  Therefore $ D_{\rho_{2}}-p^{-(j_{3}+j_{2}-2j_{1})/2}E $ is in the Mori cone.  The divisor $ D_{\rho_{2}}-p^{-(j_{3}+j_{2}-2j_{1})/2}E $ is self dual, so it is Nef.  The maximum that $ \epsilon(X_{\Sigma},D_{\rho_{2}},e) $ can be is $ p^{-(j_{3}+j_{2}-2j_{1})/2} $.  So because $ D_{\rho_{2}}-p^{-(j_{3}+j_{2}-2j_{1})/2}E $ is Nef, it must be that $ \epsilon(X_{\Sigma},D_{\rho_{2}},e) $ is equal to $ p^{-(j_{3}+j_{2}-2j_{1})/2} $.

    Since $ \mathbb{R}_{+}(D_{\rho_{2}}-p^{-(j_{3}+j_{2}-2j_{1})/2}E) $ is an extremal ray of the Nef cone and $ D_{\rho_{2}}-p^{-(j_{3}+j_{2}-2j_{1})/2} E $ is self dual, $ \mathbb{R}_{+}(D_{\rho_{2}}-p^{-(j_{3}+j_{2}-2j_{1})/2}E) $ is an extremal ray of the Mori cone.  The Mori cone is equal to the pseudo-effective cone, so $ \mathbb{R}_{+}(D_{\rho_{2}}-p^{-(j_{3}+j_{2}-2j_{1})/2}E) $ is an extremal ray of the pseudo-effective cone.  As a result, the pseudo-effective cone contains an irrational, extremal ray.

    Because the pseudo-effective cone is the support of $ \operatorname{Cox}(W) $ and the pseudo-effective cone contains an irrational, extremal ray, the Cox ring of $ W $ is not finitely generated as a $ k $-algebra by Lemma ~\ref{L:finGenSupp}.  The Cox ring of $ W $ is isomorphic to $ k[X]^{\mathbb{G}_{a}} $ by Theorem ~\ref{T:CoxRingStart} (see page \pageref{T:CoxRingStart}).  So $ k[X]^{\mathbb{G}_{a}} $ is not a finitely generated $ k $-algebra.
\end{proof}
\begin{cor}
    Let $ j_{1},j_{2},j_{3} $ be natural numbers such that
    \begin{align*}
        j_{3} & > j_{2} >j_{1}, \\
        j_{3}+j_{2} & \equiv 1 \mod{2}.
    \end{align*}
    If $ p $ is a prime natural number, and $ \beta $ is the action on $ \operatorname{Spec}(k[X]) \cong \mathbb{A}^{6}_{k} $ from Theorem ~\ref{T:mainResult} (see page \pageref{T:mainResult}) then
    \begin{equation*}
        \operatorname{Cox}(\operatorname{Bl}_{e}(\mathbb{P}(1:p^{j_{1}}+1:p^{j_{2}}+1:p^{j_{3}}+1)))\cong k[X]^{\mathbb{G}_{a}} 
    \end{equation*}
    and hence the Cox ring of $ \operatorname{Bl}_{e}(\mathbb{P}(1:p^{j_{1}}+1:p^{j_{2}}+1:p^{j_{3}}+1)) $ is not finitely generated.
\end{cor}
\begin{proof}
    We see from the proof of Theorem ~\ref{T:mainResult} (see page \pageref{T:mainResult}) that if \linebreak $ r_{1}(X),r_{2}(X),r_{3}(X) $ are the functions in ~\eqref{E:89} (see page \pageref{E:89}), then $ k[X]^{\mathbb{G}_{a}}_{x_{1}} $ is isomorphic to
    \begin{equation*}
        k[x_{1}^{\pm 1},x_{2},r_{1}(X),r_{2}(X),r_{3}(X)]
    \end{equation*}
    By Theorem ~\ref{T:CoxRingStart} (see page \pageref{T:CoxRingStart}) if $ A $ is the ring $ k[x_{2},r_{1}(X),r_{2}(X),r_{3}(X)] $, and $ \mathfrak{p} $ is the ideal $ \langle x_{1} \rangle \cap A $, then
    \begin{equation*}
        k[X]^{\mathbb{G}_{a}} \cong \oplus_{i=0}^{\infty} \mathfrak{p}^{(i)}.
    \end{equation*}
    If we endow $ A $ with the grading below:
    \begin{align}
        \deg(x_{2}) &= 1, \notag \\
        \deg(r_{1}(X)) &= p^{j_{1}}+1, \notag \\
        \deg(r_{2}(X)) &= p^{j_{2}}+1, \notag \\
        \deg(r_{3}(X)) &= p^{j_{3}}+1, \label{E:90}
    \end{align}
    then
    \begin{align*}
        k[X]^{\mathbb{G}_{a}} & \cong \oplus_{i=0}^{\infty} \mathfrak{p}^{(i)} \\
        & \cong \oplus_{i,d \in \mathbb{N}_{0}} \mathfrak{p}^{(i)} \cap A_{d}.
    \end{align*}
    The ring $ A $ with the grading in ~\eqref{E:90} is isomorphic to $ \operatorname{Cox}(\mathbb{P}(1:p^{j_{1}}+1:p^{j_{2}}+1:p^{j_{3}}+1)) $.  The scheme $ \mathcal{V}(\mathfrak{p}) $ is isomorphic to the identity of this weighted projective space.  Therefore
    \begin{align*}
        k[X]^{\mathbb{G}_{a}} & \cong \oplus_{i=0}^{\infty} \mathfrak{p}^{(i)} \\
        & \cong \oplus_{i,d \in \mathbb{N}_{0}} \mathfrak{p}^{(i)} \cap A_{d} \\
        & \cong \operatorname{Cox}(\operatorname{Bl}_{e}(\mathbb{P}(1:p^{j_{1}}+1:p^{j_{2}}+1:p^{j_{3}}+1))).
    \end{align*}
    All other assertions were proved in Theorem ~\ref{T:mainResult} (see page \pageref{T:mainResult}).
\end{proof}
\section{Appendix 1: Positivity Properties.}
We will recall some material from \cite{LazarsfeldI}.  These definitions are found on \cite[Chapter 1, Ample and Nef Line Bundles, Section 1, Preliminaries: Divisors, Line Bundles, and Linear Series, Definitions 1.1.14-1.1.15]{LazarsfeldI}.
\begin{definition}
    If $ X $ is a complete variety, then two Cartier divisors $ D_{1},D_{2} $ are \emph{numerically equivalent} if $ D_{1} \cdot C = D_{2} \cdot C $ for every irreducible curve $ C \subseteq X $.  A divisor $ D $ is \emph{numerically trivial} if it is numerically equivalent to zero.  We denote the set of numerically trivial divisors of $ X $ by $ \operatorname{Num}(X) $.  The \emph{N\'{e}ron-Severi group} of $ X $ is equal to $ \operatorname{CDiv}(X)/\operatorname{Num}(X) $.  We denote the N\'{e}ron-Severi group of $ X $ by $ N^{1}(X) $.
\end{definition}
The N\'{e}ron-Severi group is a free abelian group of finite rank (see \cite[Chapter 1, Section 1, Proposition 1.1.16]{LazarsfeldI})
\begin{definition}
    If $ X $ is a complete variety, then a Cartier divisor
    \begin{align*}
        D &\in N^{1}(X)_{\mathbb{R}} \\
        &= N^{1}(X) \otimes_{\mathbb{Z}} \mathbb{R},
    \end{align*}
    is \emph{ample} if it is equal to $ \sum_{i=1}^{\ell} c_{i}A_{i} $ where $ A_{i} $ are ample Cartier divisors and $ c_{i} \in \mathbb{R}_{+} $.
\end{definition}
\begin{definition}
    If $ X $ is a complete variety, then a divisor
    \begin{align*}
        D & \in N^{1}(X)_{\mathbb{R}} \\
        &= N^{1}(X) \otimes_{\mathbb{Z}} \mathbb{R},
    \end{align*}
    is \emph{Nef} if $ D \cdot C \ge 0 $ for all irreducible curves $ C \subseteq X $.  The Nef cone is the convex cone of $ N^{1}(X)_{\mathbb{R}} $ which is generated by Nef divisors.
\end{definition}
The definition in \cite[Chapter II, Schemes, Section 7, Projective Morphisms, pg. 153]{HartshorneAG} does not, a' priori, suggest any relation with the definition of ampleness that we supplied above.  However, \cite[Chapter 1, Section 4, Nef Line Bundles and Divisors, Theorem 1.4.23]{LazarsfeldI} states that the interior of the Nef cone is the cone of ample divisors.
\begin{definition}
    Let $ X $ be a complete variety.  We denote the set of finite, $ \mathbb{R} $-linear combinations of irreducible curves $ C_{i} \subset X $ by $ Z_{1}(X)_{\mathbb{R}} $.  If $ \gamma_{1},\gamma_{2} \in Z_{1}(X)_{\mathbb{R}} $ have the property that $ D \cdot \gamma_{1} = D \cdot \gamma_{2} $ for all $ D \in \operatorname{Div}(X)_{\mathbb{R}} $, then $ \gamma_{1} $ and $ \gamma_{2} $ are \emph{numerically equivalent}.  We denote the vector space $ Z_{1}(X)_{\mathbb{R}} $ modulo numerical equivalence by $ N_{1}(X)_{\mathbb{R}} $.  The intersection pairing defines a perfect pairing $ N^{1}(X)_{\mathbb{R}} \times N_{1}(X)_{\mathbb{R}} \to \mathbb{R} $.
\end{definition}
\begin{definition}
    If $ X $ is a complete variety, then the cone $ NE(X) $ is the cone spanned by the classes all effective $ 1 $-cycles in $ N_{1}(X)_{\mathbb{R}} $, i.e., $ \sum_{i=1}^{\ell} a_{i}[C_{i}] $, where $ a_{i} \in \mathbb{R}_{+} $ and $ [C_{i}] $ is the class of an effective $ 1 $-cycle in $ N_{1}(X)_{\mathbb{R}} $.  The \emph{Mori cone} of $ X $ is the closure of $ NE(X) $.
\end{definition}
The Mori cone is the dual cone to the Nef cone, i.e.,
\begin{equation*}
    \overline{NE(X)} = \, \{ \, \, \gamma \in N_{1}(X)_{\mathbb{R}} \, \, \text{such that} \, \, \gamma \cdot D \ge 0 \, \, \text{for all} \, \, D \in \operatorname{Nef}(X) \}.
\end{equation*}
\section{Appendix 2: Cox Rings}
We recall information about Cox rings that may be found in \cite{ADHA}.  While the authors of this book work in characteristic zero, as you cannot guarantee that the Cox ring will be integral domains in positive characteristic, we will not need to do so for the purposes of this paper.  In this paper, if we examine the Cox ring of a projective variety, then that projective variety will have a free class group.  By \cite[Chapter 1, Basic Concepts, Section 3, Divisorial Algebras, Subsection 1, Sheaves of Divisorial Algebras, Remark 1.3.1.4]{ADHA} the Cox ring of these varieties are integral domains.
\begin{definition}
    A normal variety $ X $ is \emph{$ \mathbb{Q} $-factorial} if for every Weil divisor $ D \in \operatorname{Cl}(X) $, there is a natural number $ \ell \in \mathbb{N} $ such that $ \ell D $ is Cartier, i.e., every Weil divisor is $ \mathbb{Q} $-Cartier.
\end{definition}
\begin{definition}
    Let $ X $ be a normal, $ \mathbb{Q} $-factorial, projective variety with a finitely generated, free class group.  The effective cone $ \operatorname{Eff}(X) $ is the cone spanned by all effective Weil divisors in $ N^{1}(X)_{\mathbb{R}} $.  The \emph{pseudo-effective cone} is the closure of the effective cone in $ N^{1}(X)_{\mathbb{R}} $.

    The \emph{Cox sheaf} is the sheaf $ \oplus_{[D] \in \operatorname{Cl}(X)} \mathcal{O}_{X}(D) $, where $ [D] $ denotes the class of a Weil divisor $ D $ in the class group of $ X $.  The \emph{Cox ring} is the ring of global sections of the Cox sheaf.  
\end{definition}
The Cox sheaf may be endowed with the structure of a sheaf of rings as follows.  Since the class group is free, there exists a section of the quotient map $ \pi: \operatorname{WDiv}(X) \to \operatorname{Cl}(X) $.  Choose such a section.  If $ s_{1} \in H^{0}(U, \mathcal{O}_{X}(D_{1})) $ and $ s_{2} \in H^{0}(U, \mathcal{O}_{X}(D_{2})) $, then $ s_{1} \cdot s_{2} $ is mapped to the corresponding element of $ H^{0}(U, \mathcal{O}_{X}(D_{1}+D_{2})) $.  This appears at first to depend on the choice of section of $ \pi: \operatorname{WDiv}(X) \to \operatorname{Cl}(X) $.  In the Lemma below we show that the Cox sheaf itself is independent of the choice of section.

We summarize the proof that this construction of the Cox ring and Cox sheaf does not depend on the section from the Class group to the group of Weil divisors.  This proof may be found in \cite{ADHA}.
\begin{lem}
    The construction of the Cox sheaf of a variety with a free class group is independent of the choice of a section of the homomorphism $ \pi: \operatorname{WDiv} \to \operatorname{Cl}(X) $.
\end{lem}
\begin{proof}
    We should remark that because $ X $ is regular in codimension one, and $ \mathbb{Q} $-factorial, $ \mathcal{O}_{X}(D_{1}) \mathcal{O}_{X}(D_{2}) = \mathcal{O}_{X}(D_{1}+D_{2}) $.

    Because the class group of $ X $ is free, there is an $ \ell \in \mathbb{N} $ such that the class group of $ X $ is isomorphic to $ \mathbb{Z}^{\ell} $.  To define a section of the group homomorphism $ \pi $ one needs to find a subset $ K $ of $ \operatorname{WDiv}(X) $ generated by $ \ell $ divisors $ D_{1},\dots,D_{\ell} $ such that $ \pi \mid_{K} $ maps surjectively onto $ \operatorname{Cl}(X) $.

    If $ K $ and $ K_{1} $ are subsets of $ \operatorname{WDiv}(X) $ such that both $ K $ and $ K_{1} $ map surjectively onto $ \operatorname{Cl}(X) $, then let $ \mathcal{R} $ be the sheaf of rings $ \oplus_{D \in K} \mathcal{O}_{X}(D) $ and $ \mathcal{R}_{1} $ be the sheaf of rings $ \oplus_{D \in K_{1}} \mathcal{O}_{X}(D) $.

    If $ D_{1},\dots,D_{\ell} $ are the elements of $ K $ and $ E_{1},\dots,E_{\ell} $ are the elements of $ K_{1} $, then there are elements $ f_{i} $ such that $ D_{i}-\operatorname{div}(f_{i}) = E_{i} $.   Let us define a homomorphism $ \eta $ from $ K $ to the field of fractions of $ X $ via
    \begin{equation*}
        \eta(\sum_{i=1}^{\ell} a_{i}D_{i}) = \prod_{i=1}^{\ell} f_{i}^{a_{i}}.
    \end{equation*}

    A divisor $ D $ equal to $ \sum_{i=1}^{\ell}a_{i}D_{i} $ has the property that the class of $ D $ is equal to the class of $ D-\operatorname{div}(\eta(D)) $ in the class group of $ X $.  The divisor $ D-\eta(D) $ may be written in terms of the generators of $ K_{1} $.  As a result, we may map an element $ f \in \Gamma(U, \mathcal{O}_{X}(D)) \subset \mathcal{R}(U) $ to the element $ \eta(D)f \in \mathcal{R}_{1}(U) $.  This is an isomorphism of sheaves.  As a result, the Cox sheaf and Cox ring are well defined.
\end{proof}
\begin{definition} \label{D:support}
    Let $ B $ be a $ \mathbb{Z}^{\ell} $-graded ring.  The \emph{support} of $ B $ is the semigroup generated by $ (i_{1},\dots,i_{\ell}) \in \mathbb{Z}^{\ell} $ such that $ B_{I} \ne 0 $.
\end{definition}
\begin{lem} \label{L:finGenSupp}
    If a $ \mathbb{Z}^{\ell} $-graded ring $ B $ is an integrally closed, finitely generated $ B_{(0,\dots,0)} $-algebra, then the support of $ B $ (see Definition ~\ref{D:support} on page \pageref{D:support}) is a finitely generated semigroup, i.e, every element is the integral, linear combination of its generators.  Moreover, there is a rational, polyhedral cone $ \sigma $ such that $ \sigma \cap \mathbb{Z}^{\ell} $ is equal to the support of $ B $.
\end{lem}
\begin{proof}
    Because $ B $ is finitely generated, it is generated by a finite number of elements $ \{y_{1},\dots,y_{n}\} $.  Each element $ y_{i} $ is the sum of its homogeneous parts.  Therefore, there are homogeneous elements $ y_{i,I} $ (with respect to the $ \mathbb{Z}^{\ell} $-grading) such that $ y_{i} $ is equal to $ \sum_{I} y_{i,I} $.  Because $ y_{i,I} \in B $,
    \begin{equation*}
        B_{(0,\dots,0)}[y_{1},\dots,y_{n}] \cong B_{(0,\dots,0)}[\{y_{i,I}\}].
    \end{equation*}
    Therefore, we may assume that there are elements $ y_{1},\dots,y_{n} $ such that the multi-degree of $ y_{i} $ is $ J_{i} $ and $ B $ is equal to $ B_{(0,\dots,0)}[y_{1},\dots,y_{n}] $ as a $ B_{(0,\dots,0)} $ algebra.  We claim that the support of $ B $ is equal to $ \left(\mathbb{R}_{+} J_{1}+\dots+\mathbb{R}_{+} J_{n}\right) \cap \mathbb{Z}^{\ell} $.

    If $ x $ is an element of $ B $, then $ x $ is equal to $ \sum_{I} r_{I}(y_{1},\dots,y_{n}) $ where \linebreak $ r_{I}(y_{1},\dots,y_{n}) $ is homogeneous of multi-degree $ I $.  If we can show that $ I $ is the integral, linear combination of $ J_{1},\dots,J_{n} $ for any $ I $ such that $ r_{I}(y_{1},\dots,y_{n}) $ is non-zero, then it will be clear that $ \left(\mathbb{N}_{0}J_{1}+\dots+\mathbb{N}_{0}J_{n}\right) \cap \mathbb{Z}^{\ell} $ is equal to the support of $ B $.  Because $ r_{I}(y_{1},\dots,y_{n}) $ is a homogeneous polynomial with respect to the multi-grading, its multi-degree is equal to the multi-degree of any monomial $ \prod_{i=1}^{n} y_{i}^{s_{i}} $ with a non-zero coefficient in the expansion of $ r_{I}(y_{1},\dots,y_{n}) $.  If $ \prod_{i=1}^{n} y_{i}^{s_{i}} $ is one such monomial, then $ I $ is equal to $ \sum_{i=1}^{n} s_{i} J_{i} $.  As a result, $ I $ is the integral, linear combination of $ J_{1},\dots,J_{n} $ if $ r_{I}(y_{1},\dots,y_{n}) $ is non-zero.
    
    Let $ \sigma $ equal $ \mathbb{R}_{+}J_{1}+\dots \mathbb{R}_{+}J_{n} $, or in the language of toric varieties $ \operatorname{Cone}(J_{1},\dots,J_{n}) $.  Since $ \operatorname{Spec}(B) $ is normal, it is equal to the toric variety $ \operatorname{Spec}(k[\sigma \cap \mathbb{Z}^{\ell}]) $.  So the cone $ \sigma $ is saturated and $ \sigma \cap \mathbb{Z}^{\ell} = \mathbb{N}_{0} J_{1}+\cdots +\mathbb{N}_{0}J_{n} $.
\end{proof}
The support of the Cox ring of a variety $ Z $ with free class group is the pseudo-effective cone.
\section{Appendix 3: Toric Varieties.} \label{S:ToricVarieties}
We recall the following material from ~\cite{CoxLittleSchenck} with only minor changes in notation.  This is not intended to be a complete introduction to the subject.
\begin{definition}
    A \emph{toric variety} is a variety $ X $ such that there is an equivariant embedding of a torus $ T_{N} \cong \mathbb{G}_{m}^{n} $ as a Zariski open subset of $ X $.
\end{definition}
We would like to emphasize that many of the constructions we shall detail depend on a specific identification of the torus.  This identification is not a canonical one.  Namely if one acts by an automorphism of the underlying toric variety, then the identity of the torus, the points of the torus, and many other pieces of data that we have yet to talk about change with respect to the underlying automorphism.

Two important pieces of data associated to a torus are the \emph{lattice of characters} (denoted $ M $) of the torus and \emph{the lattice of one parameter sub-groups} (denoted $ N $).  A character of a linear algebraic group $ G $ is a morphism of linear algebraic groups $ \chi: G \to \mathbb{G}_{m} $.  If $ G $ is the group $ \mathbb{G}_{m}^{n} $, then a character $ \chi: \mathbb{G}_{m}^{n} \to \mathbb{G}_{m} $ maps $ (z_{1},\dots,z_{n}) $ to $ \prod_{i=1}^{n} z_{i}^{a_{i}} $.  As a result, this character is determined by
\begin{align*}
     (a_{1},\dots,a_{n}) & \in \mathbb{Z}^{n} \\
       & = M.
\end{align*}
A one parameter sub-group of a linear algebraic group $ G $ is a morphism of linear algebraic groups $ \lambda: \mathbb{G}_{m} \to G $.  If $ G $ is an $ n $-dimensional torus $ \mathbb{G}_{m}^{n} $, then $ \lambda(z) $ is equal to $ (z^{b_{1}},\dots,z^{b_{n}}) $.  So, the one parameter sub-group $ \lambda $ is determined by
\begin{align*}
    (b_{1},\dots,b_{n}) &\in \mathbb{Z}^{n} \\
    & = N.
\end{align*}
If $ \lambda^{u} $ is a one parameter sub-group determined by $ (u_{1},\dots,u_{n}) \in \mathbb{Z}^{n} $ and $ \chi^{m} $ is a character determined by $ (m_{1},\dots,m_{n}) \in \mathbb{Z}^{n} $, then
\begin{align*}
    \chi^{m} \circ \lambda^{u}(z) &= \chi^{m}(z^{u_{1}},\dots,z^{u_{n}}) \\
    &= \prod_{i=1}^{n} z^{m_{i}u_{i}} \\
    &= z^{\sum_{i=1}^{n} m_{i}u_{i}} \\
    &= z^{\langle m,u \rangle},
\end{align*}
where $ \langle \cdot,\cdot \rangle $ is the standard inner product on $ \mathbb{R}^{n} $.  The standard inner product maps $ M \times N $ to $ \mathbb{Z} $ and identifies $ M $ with $ \operatorname{Hom}_{\mathbb{Z}}(N, \mathbb{Z}) $ and $ N $ with $ \operatorname{Hom}_{\mathbb{Z}}(M,\mathbb{Z}) $.

This identification places $ M $ within
\begin{align*}
     M_{\mathbb{R}} &= (M \otimes_{\mathbb{Z}} \mathbb{R}) \\
     & \cong \mathbb{R}^{n},
\end{align*}
and $ N $ within
\begin{align*}
     N_{\mathbb{R}} &= (N \otimes_{\mathbb{Z}} \mathbb{R}) \\
     & \cong \mathbb{R}^{n}.
\end{align*}
If $ S $ is a finite sub-set of $ N_{\mathbb{R}} $ (respectively $ M_{\mathbb{R}} $), then a \emph{convex polyhedral cone} $ \sigma $ is a set $ \sum_{u \in S} \mathbb{R}_{+}u $, which we shall denote by $ \operatorname{Cone}(S) $.  If $ m \in M_{\mathbb{R}} $, then a \emph{hyperplane} $ H_{m} $ is the set of $ u \in N_{\mathbb{R}} $ such that $ \langle m, u \rangle = 0 $.  A convex, polyhedral cone $ \sigma $ is \emph{rational} if it is possible to find generators $ u_{1},\dots,u_{\ell} \in N $ such that $ \sigma $ is equal to $ \operatorname{Cone}(u_{1},\dots,u_{\ell}) $.  A \emph{face} of a convex, polyhedral cone $ \sigma $ is a cone equal to $ \sigma \cap H_{m} $ for $ m \in M_{\mathbb{R}} $.  A polyhedral cone is \emph{strongly convex} if the origin is a face.

If $ \sigma \subseteq N_{\mathbb{R}} $ is a convex, polyhedral cone, then the \emph{dual cone} $ \sigma^{\vee} $ is the set of points $ m \in M_{\mathbb{R}} $ such that $ \langle m, u \rangle \ge 0 $ for all $ u \in \sigma $.  It is also useful to denote the set of points $ m \in M_{\mathbb{R}} $ such that $ \langle m, u \rangle = 0 $ for all $ u \in \sigma $ by $ \sigma^{\perp} $.  If $ \sigma $ is a strongly convex, rational, polyhedral cone in $ N_{\mathbb{R}} $, then $ \operatorname{Spec}(k[\sigma^{\vee} \cap M]) $ is a normal, affine, toric variety.  In fact, all normal, affine, toric varieties have this form.  Many books denote this affine toric variety by $ U_{\sigma} $.

It is a standard result that closed points of an affine scheme correspond to maximal ideals of the affine coordinate ring.  In the case of a normal, affine, toric variety, these correspond to semigroup homomorphisms from $ \sigma^{\vee} \cap M $ to $ k $, where $ k $ is a semigroup under multiplication.  If $ \tau: \sigma^{\vee} \cap M \to k $ is a semigroup homomorphism, then one obtains a ring homomorphism $ \tau^{\sharp}: k[\sigma^{\vee} \cap M] \to k $, which is the extension by linearity of the map which sends $ \chi^{m} $ to $ \tau(m) $.  If $ x $ is a closed point of $ U_{\sigma} $, then one may define a semigroup homomorphism from $ \sigma^{\vee} \cap M $ to $ k $ by the map $ \tau_{x} $ which sends $ m \in \sigma^{\vee} \cap M $ to $ \chi^{m}(x) $.

If $ \sigma \subseteq N_{\mathbb{R}} $ is a strongly convex, rational, polyhedral cone, then \emph{the span of $ \sigma $} is the real vector space $ \sigma +(-\sigma) $, i.e., it is the real vector space spanned by elements of the form $ u_{1}+u_{2} $ where $ u_{1} \in \sigma $ and $ u_{2} \in -\sigma $.  The dimension of a cone $ \sigma $ is the dimension of its span.  A point $ u \in \sigma $ is in the \emph{relative interior} of a cone $ \sigma $ if it is in the interior of the span of $ \sigma $.  We denote the relative interior of $ \sigma $ by $ \operatorname{Relint}(\sigma) $.

If $ \sigma $ is a strongly convex, rational, polyhedral cone of $ N_{\mathbb{R}} $, then there is a ``distinguished point'' $ \gamma_{\sigma} $ corresponding to the semigroup homomorphism which sends a character $ m \in \sigma^{\vee} \cap M $ to $ 1 $ if $ m \in \sigma^{\perp} \cap M $ and sends it to zero if it is not in $ \sigma^{\perp} \cap M $.  We should remark why this is a semigroup homomorphism.  If $ m_{1},m_{2} \in \sigma^{\vee} \cap M $ and $ m_{1}+m_{2} \in \sigma^{\perp} \cap M $, then we claim that both $ m_{1},m_{2} \in \sigma^{\perp} \cap M $.  If $ m_{1}+m_{2} \in \sigma^{\perp} \cap M $, then
\begin{align*}
    0 &= \langle m_{1}+m_{2},u \rangle \\
    &= \langle m_{1},u \rangle +\langle m_{2},u \rangle.
\end{align*}
For all $ u \in \sigma $, the inner product of $ m_{1} $ and $ u $ is greater than or equal to zero because $ m_{1} \in \sigma^{\vee} \cap m $.  If $ \langle m_{1},u \rangle >0 $, then $ \langle m_{2},u \rangle <0 $ in order for $ \langle m_{1}+m_{2},u \rangle $ to equal zero.  However, if $ \langle m_{2},u \rangle <0 $, then this contradicts the fact that $ m_{2} \in \sigma^{\vee} \cap m $.  So if $ m_{1},m_{2} \in \sigma^{\vee} \cap M $ and $ m_{1}+m_{2} \in \sigma^{\perp} \cap M $, then $ m_{1},m_{2} \in \sigma^{\perp} \cap M $.  Therefore, the map from $ \sigma^{\vee} \cap M $ to $ k $ which sends a character of $ \sigma^{\perp} \cap M $ to $ 1 $ and all other characters of $ \sigma^{\vee} \cap M $ to zero is a well defined semigroup morphism.

The above description tells us that $ \gamma_{\sigma} $ is a well defined point of $ U_{\sigma} $.  However, what does this point look like?  If $ \sigma $ is an $ \ell $ dimensional cone, then there exist $ u_{1},\dots,u_{\ell} \in \sigma $ such that $ \{u_{1},\dots,u_{\ell}\} $ is a basis of the span of $ \sigma $.  One may extend this basis of the span of $ \sigma $ to a basis $ \{u_{1},\dots,u_{n}\} $ of $ N_{\mathbb{R}} $.  The semigroup $ \sigma^{\vee} \cap M $ is generated by $ u_{1}^{\ast},\dots,u_{\ell}^{\ast},\pm u_{\ell+1}^{\ast},\dots, \pm u_{n}^{\ast} $.  The semigroup $ \sigma^{\perp} $ is generated by $ \pm u_{\ell+1}^{\ast},\dots,\pm u_{n}^{\ast} $.  With respect to this basis, the point $ \gamma_{\sigma} $ has the form below:
\begin{align*}
    \gamma_{\sigma} &= (\gamma_{\sigma}(u_{1}^{\ast}),\dots,\gamma_{\sigma}(u_{n}^{\ast})) \\
    &=(\underbrace{0,\dots,0}_{\ell}, \underbrace{1,\dots,1}_{n-\ell}).
\end{align*}

The distinguished point of the strongly convex, rational, polyhedral cone $ \sigma $ depends on the identification of the torus of the underlying toric variety, the lattice of characters and one parameter sub-groups, and the cone $ \sigma $.  One may ask what ``distinguishes'' the distinguished point of the cone $ \sigma $.  The answer is the following.  Because $ \sigma^{\perp} \cap M $ is not just a semigroup but a group, $ \operatorname{Hom}_{\mathbb{Z}}(\sigma^{\perp} \cap M, k^{\ast}) $ is a torus.  If $ \sigma $ is equal to $ \operatorname{Cone}(u_{1},\dots,u_{\ell}) $ and we extend $ u_{1},\dots,u_{\ell} $ to a basis $ \{u_{1},\dots,u_{n}\} $ of $ N_{\mathbb{R}} $, then $ \operatorname{Hom}_{\mathbb{Z}}(\sigma^{\perp} \cap M, k^{\ast}) $ is an $ n-\ell $-dimensional torus.  If $ M $ is the lattice of characters then the torus of the underlying toric variety is identified with $ \operatorname{Hom}_{\mathbb{Z}}(M,k^{\ast}) $.  As we have said before, this identification is not canonical.  There is an $ \ell $-dimensional sub-torus of $ \operatorname{Hom}_{\mathbb{Z}}(M, k^{\ast}) $ which acts trivially on the point $ \gamma_{\sigma} $.  The torus $ \operatorname{Hom}_{\mathbb{Z}}(\sigma^{\perp}\cap M, k^{\ast}) $ is the quotient of $ \operatorname{Hom}_{\mathbb{Z}}(M,k^{\ast}) $ by this sub-torus, and $ \gamma_{\sigma} $ is the identity of $ \operatorname{Hom}_{\mathbb{Z}}(\sigma^{\perp} \cap M, k^{\ast}) $.

Is it possible to construct a normal toric variety purely from the lattices $ N $ and $ M $ and a set of cones $ \Sigma $ with certain properties?  Namely, if one had a collection of cones $ \Sigma $ contained in $ N_{\mathbb{R}} \cong \mathbb{R}^{n} $, then could one glue together the varieties $ U_{\sigma} $ and obtain a normal, toric variety provided certain conditions on the collection $ \Sigma $ made sure that the gluing was well defined?  The advantage here would be that one would not need to know that an equivariant embedding of the torus existed.  Everything could be obtained from the collection of cones $ \Sigma $ and the lattices $ N $ and $ M $.  The answer to this question is yes.
\begin{definition}
    Let $ N \cong \mathbb{Z}^{n} $ be the lattice of one parameter subgroups of a torus $ T_{N} $ and $ M $ the dual lattice of characters.  Likewise, let
    \begin{align*}
        N_{\mathbb{R}} & = N \otimes_{\mathbb{Z}} \mathbb{R} \\
        & \cong \mathbb{R}^{n},
    \end{align*}
    and
    \begin{align*}
        M_{\mathbb{R}} & = M \otimes_{\mathbb{Z}} \mathbb{R} \\
        & \cong \mathbb{R}^{n},
    \end{align*}
    be the extension of $ N $ and $ M $ to $ n $-dimensional, real, vector spaces.   A \emph{fan} $ \Sigma $ in $ N_{\mathbb{R}} $ is a finite collection of strongly convex, rational, polyhedral cones $ \sigma \subseteq N_{\mathbb{R}} $ such that
    \begin{itemize}
        \item[i)] For all $ \sigma \in \Sigma $, each face of $ \sigma $ is an element of $ \Sigma $,
        \item[ii)] For all $ \sigma_{1},\sigma_{2} \in \Sigma $, the intersection $ \sigma_{1} \cap \sigma_{2} \in \Sigma $.
    \end{itemize}
\end{definition}
The conditions in the definition of a fan show that if $ \Sigma $ is a fan in $ N_{\mathbb{R}} $, then one obtains a normal, toric variety $ X_{\Sigma} \cong \left(\coprod_{\sigma \in \Sigma} U_{\sigma} \right)/\sim $.  This allows one to obtain a toric variety purely from the lattices $ N $ and $ M $ along with the data of a fan $ \Sigma $.  We denote the torus of a toric variety $ X_{\Sigma} $ by $ T_{N} $.

One finds the following theorem in \cite[Chapter 3, Normal Toric Varieties, Section 2, The Orbit Cone Correspondence, Proposition 3.2.2]{CoxLittleSchenck}:
\begin{quote}
    Let $ \sigma \subseteq N_{\mathbb{R}} $ be a strongly convex, rational, polyhedral cone and $ u \in N $.  Then
    \begin{equation*}
        u \in \sigma \Leftrightarrow \lim_{t \to 0} \lambda^{u}(t) \quad \text{exists in } U_{\sigma}.
    \end{equation*}
    Moreover, if $ u \in \operatorname{Relint}(\sigma) $, then $ \lim_{t \to 0} \lambda^{u}(t) = \gamma_{\sigma} $
\end{quote}
If $ \sigma $ is a strongly convex, rational, polyhedral cone, then we denote the $ T_{N} $ orbit of $ \gamma_{\sigma} $ by $ O(\sigma) $.  We denote the closure of $ O(\sigma) $ in $ X_{\Sigma} $ by $ V(\sigma) $.  The varieties $ V(\sigma) $ are the only torus invariant, closed, sub-varieties of $ X_{\Sigma} $ by the Orbit-Cone correspondence, part of which is stated below (see \cite[Chapter 3, Section 2, Theorem 3.2.6 a),b)]{CoxLittleSchenck}):
\begin{quote}
    Let $ X_{\Sigma} $ be the toric variety of the fan $ \Sigma $ in $ N_{\mathbb{R}} $.  Then:
    \begin{itemize}
        \item[a)] There is a bijective correspondence
            \begin{align*}
                \{ \text{cones} \, \sigma \, \text{in} \, \Sigma \} & \leftrightarrow \{T_{N}-\text{orbits in }\, X_{\Sigma}\} \\
                \sigma & \leftrightarrow O(\sigma) \cong \operatorname{Hom}_{\mathbb{Z}}(\sigma^{\perp} \cap M, k^{\ast}).
            \end{align*}
        \item[b)] Let $ n= \dim(N_{\mathbb{R}}) $.  For each cone $ \sigma \in \Sigma $, $ \dim(O(\sigma)) = n-\dim(\sigma) $.
    \end{itemize}
\end{quote}
A consequence of the Orbit-Cone correspondence is that all $ n-\ell $-dimensional, closed, torus invariant, sub-varieties of $ X_{\Sigma} $ correspond to $ \ell $-dimensional cones of $ \Sigma $.  This correspondence functions as follows.  Given an $ \ell $-dimensional cone $ \sigma \in \Sigma $, there exists a distinguished point $ \gamma_{\sigma} $.  The orbit of this point is a locally closed, $ n-\ell $-dimensional, torus invariant, sub-variety of $ X_{\Sigma} $ which we denote by $ O(\sigma) $.  The closure of this sub-variety is an $ n-\ell $-dimensional, torus invariant, closed, sub-variety $ V(\sigma) $.  If $ N_{\sigma} $ is the sub-lattice of $ N $ whose points are contained in the span of $ \sigma $, then let $ N(\sigma) $ equal $ N/N_{\sigma} $.  In \cite[Chapter 3, Section 2, Lemma 3.2.4]{CoxLittleSchenck} Cox, Little and Schenck prove that there is a perfect pairing between $ \sigma^{\perp} \cap M $ and $ N(\sigma) $ obtained from the original pairing between $ M $ and $ N $.  With respect to this pairing
\begin{align*}
    \operatorname{Hom}_{\mathbb{Z}}(\sigma^{\perp} \cap M, k^{\ast}) & \cong N(\sigma) \otimes_{\mathbb{Z}} k^{\ast} \\
    & \cong T_{N(\sigma)}.
\end{align*}
The correspondence between cones $ \sigma \in \Sigma $ and $ T_{N} $ orbits $ O(\sigma) $ is bijective.

We will be very interested in the class group and Picard groups of toric varieties.  The orbit cone correspondence shows that divisors of a toric variety correspond to one dimensional cones of the fan $ \Sigma $.  A one dimensional cone of the fan $ \Sigma $ is called \emph{a ray}.  We denote the set of $ r $-dimensional cones by $ \Sigma(r) $.  When $ \rho \in \Sigma(1) $, then we will denote the variety $ V(\rho) $ by $ D_{\rho} $.

\begin{definition}
    A strongly convex, rational, polyhedral cone is \emph{simplicial} if its minimal generators are linearly independent over $ \mathbb{R} $.
\end{definition}
The \emph{support} of a fan $ \Sigma $ is the set $ \cup_{\sigma \in \Sigma} \sigma \subseteq N_{\mathbb{R}} $.  We denote the support of $ \Sigma $ by $  \mid \Sigma \mid $.  A normal toric variety is complete if and only if its support is equal to $ N_{\mathbb{R}} $.  We give a different proof than the one listed in \cite{CoxLittleSchenck} that does not depend on the characteristic of the base field.  This proof is very similar to the one in \cite[Chapter 2, Singularities and Compactness, Section 4, Compactness and Properness]{Fulton}.
\begin{prop} \label{P:toricComplete}
    If $ X_{\Sigma} $ is a normal, toric variety over a field $ k $ of arbitrary characteristic, then $ X_{\Sigma} $ is complete if and only if the support of $ \Sigma $ is all of $ N_{\mathbb{R}} $.
\end{prop}
\begin{proof}
    We shall prove the Proposition by proving a more general statement.  We claim that a morphism of toric varieties $ \phi: X_{\Sigma_{1}} \to X_{\Sigma} $ is proper if and only if $ \phi^{-1}(\mid \Sigma \mid) = \mid \Sigma_{1} \mid $.

    Suppose that $ u_{1} \in N_{1} $ and $ \phi(u_{1}) \in \sigma $ for $ \sigma \in \Sigma $, but $ u_{1} \notin \mid \Sigma_{1} \mid $.  If $ \phi $ is proper, then the following diagram commutes:
    \begin{equation*}
    \xymatrix{
        \mathbb{G}_{m} \ar[rr]^{\lambda^{u_{1}}} \ar[d] & & X_{\Sigma_{1}} \ar[d]^{\phi} \\
        \mathbb{A}^{1}_{k} \ar[rr] \ar@{-->}[urr]^{\exists !} & & X_{\Sigma}
        }.
    \end{equation*}
    So the valuative criterion of properness states that $ \lim_{z \to 0} \lambda^{u_{1}}(z) $ should exist.  However, \cite[Chapter 3, Normal Toric Varieties, Section 2, The Orbit-Cone Correspondence, Proposition 3.2.2]{CoxLittleSchenck} states that this limit exists if and only if $ u_{1} \in \sigma_{1} $ for some $ \sigma_{1} \in \Sigma_{1} $.  This shows that if $ \mid \Sigma_{1} \mid \subsetneq \phi^{-1}(\mid \Sigma \mid) $ then $ \phi $ is not proper.  The contrapositive of this statement is that if $ \phi $ is proper, then $ \phi^{-1}(\mid \Sigma \mid) = \mid \Sigma_{1} \mid $.

    Assume that $ R $ is a discrete valuation ring, $ \phi^{-1}(\mid \Sigma \mid) = \mid \Sigma_{1} \mid $, and that the following diagram commutes:
    \begin{equation*}
    \xymatrix{
         \operatorname{Spec}(\operatorname{Frac}(R)) \ar[d] \ar[rr] & & X_{\Sigma_{1}} \ar[d]^{\phi} \\
         \operatorname{Spec}(R) \ar[rr] & & X_{\Sigma}
         }.
    \end{equation*}
    Also assume that the image of $ \operatorname{Spec}(R) $ is in $ U_{\sigma} \subseteq X_{\Sigma} $.  The morphism from \linebreak $ \operatorname{Spec}(\operatorname{Frac}(R)) $ to $ X_{\Sigma_{1}} $ is obtained from a semigroup morphism $ \alpha: M_{1} \to \operatorname{Frac}(R)^{\ast} $.  If $ \nu $ is the valuation on $ \operatorname{Frac}(R) $ for such that $ x \in R $ if and only if $ \nu(x) \in \mathbb{N}_{0} $, then because the image of $ \operatorname{Spec}(R) $ is in $ U_{\sigma} $,
    \begin{equation*}
        \nu(\alpha \circ \phi^{-1}(\sigma^{\vee} \cap M)) \subseteq \mathbb{N}_{0}.
    \end{equation*}
    By the assumption that $ \phi^{-1}(\mid \Sigma \mid) = \mid \Sigma_{1} \mid $, there is a cone $ \sigma_{1} \in \Sigma_{1} $ such that $ \phi(\sigma_{1}) \subseteq \sigma $ and $ \nu \circ \alpha(\sigma_{1} \cap M_{1}) \subseteq \mathbb{N}_{0} $.  So there is a map $ f: \operatorname{Spec}(R) \to X_{\Sigma_{1}} $ such that the following diagram commutes:
    \begin{equation*}
    \xymatrix{
         \operatorname{Spec}(\operatorname{Frac}(R)) \ar[d] \ar[rr] & & X_{\Sigma_{1}} \ar[d]^{\phi} \\
         \operatorname{Spec}(R) \ar@{-->}[urr]^{f} \ar[rr] & & X_{\Sigma}
         }.
    \end{equation*}
    Since toric varieties are separated, $ f $ is unique.  Therefore, $ \phi $ is proper by the valuative criterion of properness.
\end{proof}
\begin{definition}
    A toric variety $ X_{\Sigma} $ \emph{has torus factors} if $ X_{\Sigma} \cong \mathbb{G}_{m}^{\ell} \times X_{\Sigma^{'}} $ for some $ \ell \in \mathbb{N} $ and fan $ \Sigma^{'} $.
\end{definition}
Every divisor of a toric variety is linearly equivalent to a torus invariant divisor.  The set of torus invariant divisors is isomorphic to $ \oplus_{\rho \in \Sigma(1)} \mathbb{Z} D_{\rho} $ by the Orbit-Cone correspondence.  The \emph{primitive ray generator} of a ray $ \rho \in \Sigma(1) $ is the minimal element $ u_{\rho} \in \rho \cap N $.  If $ u_{\rho} $ is the primitive ray generator of $ \rho $, then the set of principal, torus invariant divisors is equal to the set of characters under the identification which sends $ m $ to $ \sum_{\rho \in \Sigma(1)} \langle m, u_{\rho} \rangle D_{\rho} $.

If a toric variety $ X_{\Sigma} $ has no torus factors, then the following sequence is exact (see \cite[Chapter 4, Divisors on Toric Varieties, Section 1, Weil Divisors on Toric Varieties, Theorem 4.1.3]{CoxLittleSchenck}):
\begin{equation} \label{E:68}
\xymatrix{
    0 \ar[r] & M \ar[r] & \oplus_{\rho \in \Sigma(1)} \mathbb{Z} D_{\rho} \ar[r] & \operatorname{Cl}(X_{\Sigma}) \ar[r] & 0
    }.
\end{equation}

For a normal, toric variety $ X_{\Sigma} $, the class group is the same as the Picard group if and only if $ X_{\Sigma} $ is smooth \cite[Chapter 4, Section 2, Cartier Divisors on Toric Varieties, Proposition 4.2.6]{CoxLittleSchenck} and a normal, toric variety is $ \mathbb{Q} $-factorial if and only if $ X_{\Sigma} $ is simplicial \cite[Chapter 4, Section 2, Proposition 4.2.7]{CoxLittleSchenck}.  If $ X_{\Sigma} $ is a normal, toric variety with fan $ \Sigma $, and $ D =\sum_{\rho \in \Sigma(1)} a_{\rho} D_{\rho} $, then \cite[Chapter 4, Section 2, Theorem 4.2.8]{CoxLittleSchenck} states that the following are equivalent:
\begin{itemize}
    \item[a)] $ D $ is Cartier,
    \item[b)] $ D $ is principle on the open set $ U_{\sigma} $ for all $ \sigma \in \Sigma $,
    \item[c)] For each $ \sigma \in \Sigma $, there is $ m_{\sigma} \in M $ with $ \langle m_{\sigma}, u_{\rho} \rangle =-a_{\rho} $ for all $ \rho \in \sigma(1) $,
    \item[d)] For each $ \sigma \in \Sigma_{\max} $ there is $ m_{\sigma} \in M $ with $ \langle m_{\sigma}, u_{\rho} \rangle = -a_{\rho} $ for all $ \rho \in \sigma(1) $.
\end{itemize}
The collection $ \{m_{\sigma}\}_{\sigma \in \Sigma} $ is the Cartier data of $ D $.
\begin{definition}
    Let $ \Sigma $ be a fan of $ N_{\mathbb{R}} $.  A \emph{support function} is a function $ \phi: \mid \Sigma \mid \to \mathbb{R} $ that is linear on each cone of $ \Sigma $.  The set of support functions is denoted $ \operatorname{SF}(\Sigma) $.  A support function $ \phi $ is \emph{integral with respect to the lattice} $ N $ if $ \phi(\mid \Sigma \mid \cap N) \subseteq \mathbb{Z} $.  The set of support functions which are integral with respect to the lattice $ N $ is denoted $ \operatorname{SF}(\Sigma,N) $.
\end{definition}
By \cite[Chapter 4, Section 2, Theorem 4.2.12]{CoxLittleSchenck} $ \operatorname{SF}(\Sigma,N) $ is isomorphic to the set of torus invariant Cartier divisors.  This theorem states:
\begin{quote}
    Let $ \Sigma $ be a fan in $ N_{\mathbb{R}} $.  Then:
    \begin{itemize}
        \item[a)] Given $ D = \sum_{\rho \in \Sigma(1)} a_{\rho} D_{\rho} $ with Cartier data $ \{m_{\sigma}\}_{\sigma \in \Sigma} $, the function
            \begin{align*}
                \phi_{D} &: \mid \Sigma \mid \to \mathbb{R} \\
                & u \mapsto \phi_{D}(u) = \langle m_{\sigma}, u\rangle \quad u \in \sigma
            \end{align*}
            is a well defined support function that is integral with respect to $ N $.
        \item[b)] $ \phi_{D}(u_{\rho}) = -a_{\rho} $ for all $ \rho \in \Sigma(1) $, so that
            \begin{equation*}
                D= - \sum_{\rho \in \Sigma(1)} \phi_{D}(u_{\rho})D_{\rho}.
            \end{equation*}
        \item[c)] The map $ D \mapsto \phi_{D} $ induces an isomorphism $ \operatorname{CDiv}_{T_{N}}(X_{\Sigma}) \cong \operatorname{SF}(\Sigma,N) $.
    \end{itemize}
\end{quote}
If $ D $ is equal to $ \sum_{\rho \in \Sigma(1)} a_{\rho} D_{\rho} $ and $ m \in M $, then $ \chi^{m} \in H^{0}(X_{\Sigma}, \mathcal{O}_{X_{\Sigma}}(D)) $ if $ \operatorname{div}(\chi^{m}) + D \ge 0 $.  This is equivalent to $ \langle m, u_{\rho} \rangle+a_{\rho} \ge 0 $ for all $ \rho \in \Sigma(1) $.  So if $ D $ is equal to $ \sum_{\rho \in \Sigma(1)} a_{\rho} D_{\rho} $, then let $ P_{D} $ be the polytope $ \{m \in M_{\mathbb{R}} $ such that $ \langle m, u_{\rho} \rangle \ge -a_{\rho} $ for all $ \rho \in \Sigma(1)\} $.  The thrust of this result is \cite[Chapter 4, Section 3, The Sheaf of a Torus Invariant Divisor, Propositon 4.3.3]{CoxLittleSchenck} which states:
\begin{quote}
    If $ D $ is a torus invariant Weil divisor on $ X_{\Sigma} $, then
    \begin{equation*}
        \Gamma(X_{\Sigma}, \mathcal{O}_{X_{\Sigma}}(D)) = \oplus_{m \in P_{D} \cap M} k \cdot \chi^{m}.
    \end{equation*}
\end{quote}

It is also possible to construct a normal toric variety $ X_{\Sigma} $ as the quotient of a quasi-affine variety by a linearly reductive group.  Let $ G $ be the group below:
\begin{equation*}
     G:=\operatorname{Hom}_{\mathbb{Z}}(\operatorname{Cl}(X_{\Sigma}), \mathbb{G}_{m}).
\end{equation*}
The points of $ G $ are $ (a_{1},\dots,a_{\operatorname{card}(\Sigma(1))}) $ such that
\begin{equation*}
     \prod_{i=1}^{\operatorname{card}(\Sigma(1))} a_{i}^{\langle m, u_{\rho_{i}} \rangle} =1, \quad \forall m \in M,
\end{equation*}
by \cite[Chapter 5, Homogeneous Coordinates on Toric Varieties, Section 1, Quotient Constructions of Toric Varieties, Lemma 5.1.1]{CoxLittleSchenck}.  Because the following sequence is exact if $ X_{\Sigma} $ has no torus factors:
\begin{equation*}
    \xymatrix{
         0 \ar[r] & M \ar[r] & \mathbb{Z}^{\operatorname{card}(\Sigma(1))} \ar[r] & \operatorname{Cl}(X_{\Sigma}) \ar[r] & 0
    }
\end{equation*}
the next sequence is exact:
\begin{equation*}
    \xymatrix{
        1 \ar[r] & G \ar[r] & \mathbb{G}_{m}^{\operatorname{card}(\Sigma(1))} \ar[r] & T_{X_{\Sigma}} \ar[r] & 1
    }
\end{equation*}

If $ \sigma $ is a strongly convex, rational, polyhedral cone of $ \Sigma $, and the affine coordinate ring of $ \mathbb{A}^{\operatorname{card}(\Sigma(1))}_{k} $ is $ k[y_{\rho_{1}},\dots,y_{\rho_{\operatorname{card}(\Sigma(1))}}] $, then let $ \widehat{\sigma} $ be the monomial $ \prod_{\rho \notin \sigma} y_{\rho} $.  The ideal $ B(\Sigma) $ is the ideal $ \langle \widehat{\sigma} \rangle_{\sigma \in \Sigma} $.  The sub-variety $ Z(\Sigma) $ is the zero set of $ B(\Sigma) $.  Another way to think of the sub-variety $ Z(\Sigma) $ is via primitive collections.

A subset $ C \subseteq \Sigma(1) $ is a primitive collection if $ C \not \subseteq \sigma(1) $ for all $ \sigma \in \Sigma $, and for every proper subset $ C^{'} \subseteq C $, there is a $ \sigma \in \Sigma $ such that $ C^{'} \subseteq \sigma(1) $.  The variety $ Z(\Sigma) $ is equal to $ \cup_{C} \mathcal{V}(\langle y_{\rho} \rangle_{\rho \in C}) $.

By \cite[Chapter 5, Section 1, Theorem 5.1.11)]{CoxLittleSchenck}, the variety $ X_{\Sigma} $ is isomorphic to $ \left( \mathbb{A}^{\operatorname{card}(\Sigma(1))}_{k} \setminus Z(\Sigma) \right) //G $, and it is a geometric quotient if and only if $ \Sigma $ is simplicial.  The ring $ k[y_{\rho_{1}},\dots,y_{\rho_{\operatorname{card}(\Sigma(1))}}] $ is the Cox ring of $ X_{\Sigma} $.

While computing the Nef cone usually requires a Herculean effort, it is remarkably easy to do so for a projective, simplicial, normal, toric variety.  The reader may find the following criterion for a divisor on a projective, simplicial, normal, toric variety to be Nef in \cite[Chapter 6, Line Bundles on Toric Varieties, Section 4, The Simplicial Case, Thoerem 6.4.9]{CoxLittleSchenck}:
\begin{quote}
    Let $ X_{\Sigma} $ be a projective simplicial toric variety.  Then:
    \begin{itemize}
        \item[a)] A Cartier divisor $ D $ is nef if and only if its support function $ \phi_{D} $ satisfies:
        \begin{equation*}
        \phi_{D}(u_{\rho_{1}}+\cdots+u_{\rho_{k}}) \ge \phi_{D}(u_{\rho_{1}})+\cdots + \phi_{D}(u_{\rho_{k}}),
        \end{equation*}
        for all primitive collections $ P = \{\rho_{1},\dots,\rho_{k}\} $ of $ \Sigma $.
        \item[b)] A Cartier divisor $ D $ is ample if and only if its support function $ \phi_{D} $ satisfies:
        \begin{equation*}
            \phi_{D}(u_{\rho_{1}}+\cdots+u_{\rho_{k}}) \ge \phi_{D}(u_{\rho_{1}})+\cdots+\phi_{D}(u_{\rho_{k}}),
        \end{equation*}
        for all primitive collections $ P = \{u_{\rho_{1}},\dots,u_{\rho_{k}}\} $ of $ \Sigma $.
    \end{itemize}
\end{quote}
Torus invariant $ 1 $-cycles of an $ n $-dimensional, normal, toric variety $ X_{\Sigma} $ correspond to cones $ \tau \in \Sigma(n-1) $ by the Orbit-Cone correspondence.  A \emph{wall} is a cone $ \tau \in \Sigma(n-1) $ such that $ \tau = \sigma \cap \sigma_{1} $ for maximal dimensional cones $ \sigma $ and $ \sigma_{1} $.  If $ \tau, \sigma $ and $ \sigma_{1} $ are described below:
\begin{align*}
    \sigma &= \operatorname{Cone}(u_{\rho_{1}},\dots,u_{\rho_{n}}) \\
    \sigma_{1} &= \operatorname{Cone}(u_{\rho_{2}},\dots,u_{\rho_{n+1}}) \\
    \tau &= \operatorname{Cone}(u_{\rho_{2}},\dots,u_{\rho_{n}}),
\end{align*}
then the minimal generators $ u_{\rho_{1}},\dots,u_{\rho_{n+1}} $ are linearly dependent and so there is a linear relation
\begin{equation} \label{E:69}
    \alpha u_{\rho_{1}}+ \sum_{i=1}^{n} b_{i} u_{\rho_{i}}+\beta u_{\rho_{n+1}} =0.
\end{equation}
In addition we may assume that $ \alpha,\beta>0 $.  A relation of the form in ~\eqref{E:69} is called a \emph{wall relation}.  Let us define $ \operatorname{mult}(\sigma) $ to be the index of the sub-lattice $ \mathbb{Z} u_{\rho_{1}}+\cdots +\mathbb{Z} u_{\rho_{\ell}} $ in $ \left(\mathbb{R} u_{\rho_{1}} + \cdots + \mathbb{R} u_{\rho_{\ell}} \right) \cap N $ for a strongly convex, rational, polyhedral, simplicial cone $ \sigma $ with minimal generators $ u_{\rho_{1}},\dots,u_{\rho_{\ell}} $.  If the wall relation in ~\eqref{E:69} holds, then \cite[Chapter 6, Line Bundles on Toric Varieties, Section 4, The Simplicial Case, Proposition 6.4.4]{CoxLittleSchenck} shows that:
\begin{itemize}
    \item[a)] $ D_{\rho} \cdot V(\tau) = 0 $ for all $ \rho \notin \{\rho_{1},\dots,\rho_{n+1}\} $,
    \item[b)] $ D_{\rho_{1}} \cdot V(\tau) = \operatorname{mult}(\tau)/\operatorname{mult}(\sigma) $, and $ D_{\rho_{n+1}} \cdot V(\tau) = \operatorname{mult}(\tau)/\operatorname{mult}(\sigma_{1}) $,
    \item[c)] \begin{align*}
        D_{\rho_{i}} \cdot V(\tau) &= \frac{b_{i} \operatorname{mult}(\tau)}{\alpha \operatorname{mult}(\sigma)} \\
        &= \frac{b_{i} \operatorname{mult}(\tau)}{\beta \operatorname{mult}(\sigma_{1})}.
        \end{align*}
\end{itemize}
\bibliographystyle{amsplain}
\bibliography{Counterexample022024}

\providecommand{\bysame}{\leavevmode\hbox to3em{\hrulefill}\thinspace}
\providecommand{\MR}{\relax\ifhmode\unskip\space\fi MR }
\providecommand{\MRhref}[2]{%
  \href{http://www.ams.org/mathscinet-getitem?mr=#1}{#2}
}
\providecommand{\href}[2]{#2}
\begin{thebibliography}{10}

\bibitem{ADHA}
Ivan Arzhantsev, Ulrich Derenthal, J{\"u}rgen Hausen, and Antonio Laface,
  \emph{Cox rings}, no. 144, Cambridge University Press, 2015.

\bibitem{CoxLittleSchenck}
David~A Cox, John~B Little, and Henry~K Schenck, \emph{Toric varieties},
  American Mathematical Soc., 2011.

\bibitem{Fulton}
William Fulton, \emph{Introduction to toric varieties}, no. 131, Princeton
  university press, 1993.

\bibitem{Goss}
David Goss, \emph{Basic structures of function field arithmetic}, Springer
  Science \& Business Media, 2012.

\bibitem{HartshorneAG}
Robin Hartshorne, \emph{Algebraic geometry}, Springer-Verlag, New York, 1977,
  Graduate Texts in Mathematics, No. 52. \MR{MR0463157 (57 \#3116)}

\bibitem{LazarsfeldI}
Robert~K Lazarsfeld, \emph{Positivity in algebraic geometry i: Classical
  setting: line bundles and linear series}, vol.~48, Springer Science \&
  Business Media, 2004.

\bibitem{MukaiCounterexample}
Shigeru Mukai, \emph{Counterexample to hilbert's fourteenth problem for the
  3-dimensional additive group}, Kyoto University. Research Institute for
  Mathematical Sciences [RIMS], 2001.

\bibitem{Nagata}
Masayoshi Nagata, M~Pavaman Murthy, et~al., \emph{Lectures on the fourteenth
  problem of hilbert}, vol.~31, Tata Institute of Fundamental Research Bombay,
  1965.

\bibitem{Ore}
Oystein Ore, \emph{On a special class of polynomials}, Transactions of the
  American Mathematical Society \textbf{35} (1933), no.~3, 559--584.

\bibitem{Seshadri}
Conjeevaram~S Seshadri, \emph{On a theorem of weitzenb{\"o}ck in invariant
  theory}, Journal of mathematics of Kyoto University \textbf{1} (1962), no.~3,
  403--409.

\bibitem{Weitz}
Roland Weitzenb{\"o}ck, \emph{{\"U}ber die invarianten von linearen gruppen},
  (1932).

\end{thebibliography}
\end{document}